\documentclass[10pt,a4paper,reqno]{amsart}
\usepackage{amssymb,graphicx,mathrsfs,eucal,xcolor,enumitem,esint}
\usepackage[%
    pdfauthor={Piero D'Ancona}%
    ,colorlinks=true%
    ,linkcolor=magenta%
    ,citecolor=cyan%
    ,urlcolor=blue%
    ,hyperfootnotes=false%
]{hyperref}

\numberwithin{equation}{section}
\newtheorem{theorem}{Theorem}[section]
\newtheorem{corollary}[theorem]{Corollary}
\newtheorem{lemma}[theorem]{Lemma}

\newtheorem{proposition}[theorem]{Proposition}
\theoremstyle{remark}
\newtheorem{remark}{Remark}[section]

\theoremstyle{definition}
\newtheorem{definition}[theorem]{Definition}
\DeclareMathOperator{\spt}{spt}

\newcommand{\bra}[1]{\langle #1 \rangle}

\title%
[Wave maps]%
{Global existence of small equivariant wave maps
on rotationally symmetric manifolds}
\date{\today}    %%% ''\date{}'' to omit date
\author{Piero D'Ancona}
\address{Piero D'Ancona: Unversit\`a di Roma ``La Sapienza'',
Dipartimento di Matematica, Piazzale A.~Moro 2, I-00185 Roma, Italy}
\email{dancona@mat.uniroma1.it}
\thanks{The first author was partially supported by
the Italian Project FIRB 2012: ``Dispersive
dynamics: Fourier Analysis and Variational Methods''}

\author{Qidi Zhang}
\address{Qidi Zhang: School of Science, East China University of
Science and Technology, Meilong Road 130, Shanghai, 200237, China}
\email{qidizhang@ecust.edu.cn}
\thanks{The second author was
partially supported by NSFC 11271322 and the Fundamental Research
Funds for the Central Universities.
He would like to thank the University of Roma
Sapienza for the hospitality; part of this work has been
finished during his stay in Roma.
}

\keywords{}
\begin{document}\begin{abstract}
  We introduce a class of rotationally invariant manifolds,
  which we call \emph{admissible}, on which
  the wave flow satisfies smoothing and Strichartz estimates.
  We deduce the global existence of
  equivariant wave maps
  from admissible manifolds to general targets,
  for small initial data of critical regularity $H^{\frac n2}$.
  The class of admissible manifolds includes in particular
  asymptotically flat manifolds and perturbations of
  real hyperbolic spaces $\mathbb{H}^{n}$ for $n\ge3$.
\end{abstract}
% e_preamble
\maketitle
% e_fold 000

\section{Introduction}\label{sec:introduction}

\emph{Wave maps} are functions $u:M^{1+n}\to N^{\ell}$
from a
Lorentzian manifold $(M^{1+n},h)$ to a
Riemannian manifold $(N^{\ell},g)$, which are critical points
for the functional on $M^{1+n}$ with Lagrangian density
$L(u)=Tr_{h}(u^{*}g)$, the trace with respect to the metric
$h$ of the pullback of the metric $g$ through the map $u$.
The space $M^{1+n}$ is usually called the \emph{base manifold}
and $N^{\ell}$ the \emph{target manifold}; both are assumed
to be smooth, complete and without boundary.
This notion extends to a Lorentzian setting
the usual definition of harmonic maps
between Riemannian manifolds.
Wave maps arise in several different physical theories,
and in particular they play an\
important role in general relativity.

When the base manifold
is the flat Minkowski space $\mathbb{R}\times \mathbb{R}^{n}$,
in local coordinates on the target, the
Euler-Lagrange equations for $L(u)$ reduce to
a system of derivative nonlinear wave equations
\begin{equation}\label{eq:wmsys}
  \square u^a + \Gamma^a_{bc}(u)
  \partial_\alpha u^b \partial^\alpha u^c = 0,
\end{equation}
where $\Gamma^{a}_{bc}$ are the Christoffel
symbols on $N^{\ell}$
and we use implicit summation over
repeated indices. The natural setting
is then the Cauchy problem with data at $t=0$
\begin{equation}\label{eq:wmdata}
  u(0,x)=u_0,\qquad u_t(0,x)=u_1.
\end{equation}
The data are taken in suitable $N^{\ell}$-valued Sobolev spaces
\begin{equation}\label{eq:wmsob}
  (u_{0},u_{1})\in H^{s}(\mathbb{R}^{m},N^{\ell})\times
  H^{s-1}(\mathbb{R}^{m},TN^{\ell})
\end{equation}
which can be defined as follows, if
$N^{\ell}$ is isometrically embedded in a
euclidean $\mathbb{R}^{\ell'}$:
\begin{equation}\label{Hs}
   H^s(\mathbb{R}^m;N^{\ell})
   :=
   \{v\in
     H^s(\mathbb{R}^m;\mathbb{R}^{\ell'}),
     \ v(\mathbb{R}^m)\subseteq N^{\ell}\}.
\end{equation}
Solutions belong to the space $C([0,T);H^{s})$,
with $T\le \infty$.
Starting with
\cite{Gu80-a},
\cite{GinibreVelo82-c}
Problem \eqref{eq:wmsys}, \eqref{eq:wmdata}
has been studied extensively;  see
\cite{ShatahStruwe98-a} and
\cite{DAnconaGeorgiev05-a}
for a review of the classical theory.

Since equation \eqref{eq:wmsys} is invariant for
the scaling $u(t,x)\mapsto u(\lambda t, \lambda x)$,
the critical Sobolev space for the data
corresponds to $s=\frac n2$. In dimension $n=1$
energy conservation is sufficient to prove global well
posedness, thus in the following we assume $n\ge2$.
Concerning local existence, the behaviour is rather clear;
Problem \eqref{eq:wmsys}--\eqref{eq:wmsob} is
\begin{itemize}
[noitemsep,topsep=4pt,parsep=0pt,partopsep=0pt]
  \item locally well posed if $s>\frac n2$
  (see
  \cite{KlainermanMachedon93-a},
  \cite{KlainermanSelberg97-a}). Note that classical energy
  estimates only allow to prove local existence for
  $s>\frac n2+1$, and the sharp result requires
  bilinear methods which exploit the null
  structure of the nonlinearity.
  \item ill posed if $s<\frac n2$
  (see
  \cite{Shatah88-a},
  \cite{DAnconaGeorgiev04-a},
  \cite{DAnconaGeorgiev04-b}).
\end{itemize}
The problem of global existence with small
data has been completely understood through the efforts
of many authors during the last 20 years
(see among the others
\cite{ShatahTahvildar-Zadeh94-a},
\cite{Tao01-e},
\cite{Tao01-d},
\cite{KlainermanRodnianski01-a},
\cite{ShatahStruwe02-a},
\cite{Krieger03-a},
\cite{Tataru05-a}).
The end result is that if the
initial data belong to
$H^{\frac n2}\times H^{\frac n2-1}$, and their
homogeneous
$\dot H^{\frac n2}\times\dot H^{\frac n2-1}$ norm
is sufficiently small, then there exists a
global solution, continuous with values in $H^{\frac n2}$,
for general targets.
Note that the solution also belongs to a suitable Strichartz
space (more on this below), and uniqueness holds only
under this additional constraint.

When the initial data are large, the geometry of the target
manifold comes into play, and the problem
presents additional difficulties;
in particular, blow up in finite time may occur.
For targets with positive curvature, when the
dimension of the base space is $n\ge3$,
blow up examples with self similar structure were
constructed already in
\cite{Shatah88-a},
\cite{ShatahTahvildar-Zadeh94-a}.
On the other hand, when the target is negatively curved,
the available blow up examples require $n\ge7$
\cite{CazenaveShatahTahvildar-Zadeh98-a}.

The case $n=2$ is especially interesting
since the critical norm $\dot H^{\frac n2}$ coincides
with the energy norm, which is conserved.
The general conjecture is that large solutions may
blow up for certain classes of targets with
positive curvature, while they can be continued globally for
geodesically convex targets.
In this generality the conjecture remains open,
and is being actively researched,
but it has been confirmed in several cases
and is supported by numerical evidence.
Note however that for compact targets it was
proved in
\cite{SterbenzTataru10-a},
\cite{SterbenzTataru10-b}
that solutions are global as long as the energy of the
initial data is below the energy of a minimal harmonic map
% between the manifolds
(so that all solutions
are global when such maps do not exist).
When the target is the hyperbolic space $\mathbb{H}^{2}$,
global existence was proved in
\cite{Krieger04-a}
(see also \cite{Krieger03-a}).
In a rotationally symmetric setting, global existence
for geodesically convex targets was obtained in
\cite{Grillakis91-a},
\cite{ChristodoulouTahvildar-Zadeh93-b},
\cite{ShatahTahvildar-Zadeh94-a},
\cite{Struwe03-a}, while blow up solutions
for the $\mathbb{S}^{2}$ target
were constructed and analyzed in
\cite{KriegerSchlagTataru08-a},
\cite{RodnianskiSterbenz10-a},
\cite{RaphaelRodnianski12-a}.
However, radially symmetric solutions into
the 2-sphere never blow up
\cite{Struwe02-b}. See
% \cite{ShatahStruwe98-a},
% \cite{DAnconaGeorgiev05-a}
\cite{Tataru04-a}
for additional information and detailed references.

The more general case of a nonflat base manifold has received
much less attention. If we restrict to maps defined on a
product $\mathbb{R}\times M^{n}$,
with $M^{n}$ a Riemannian manifold, the
wave map system in local coordinates \eqref{eq:wmsys}
becomes
\begin{equation}\label{eq:wmsysh}
  u^{a}_{tt}-\Delta_{M}u^{a}
  + \Gamma^a_{bc}(u)
    \partial_\alpha u^b \partial^\alpha u^c = 0,
\end{equation}
where $\Delta_{M}$ is the negative Laplace-Beltrami
operator on $M^{n}$. To our knowledge, there
are few results on \eqref{eq:wmsysh}.
In
\cite{ShatahTahvildar-Zadeh97-a}
the stability of equivariant,
stationary wave maps on $\mathbb{S}^{2}$ with values
in $\mathbb{S}^{2}$ is proved, while
\cite{Choquet-Bruhat00-a} considers the local existence
on Robertson-Walker spacetimes.
More recently, in
\cite{Lawrie13-a}
global existence of small wave maps is proved
in the case when $M^{n}=M^{4}$
is a four dimensional small perturbation of
flat $\mathbb{R}^{4}$, and
the stability of equivariant wave maps defined on
$\mathbb{H}^{2}$ is studied in
\cite{LawrieOhShahshahani14-a}.

In the present paper we initiate the study of equivariant
solutions of \eqref{eq:wmsysh} on more general
base manifolds $M^{n}$, $n\ge3$.
Our main result is the global existence of
equivariant wave maps for small data in the critical norm,
provided the base manifold belongs to a class
of manifolds which we call
\emph{admissible}.
The class of admissible manifolds is rather large, and
includes in particular asymptotically flat manifolds and
perturbations of real hyperbolic spaces;
see some examples in
Remark \ref{rem:admman} below and a more detailed discussion
in Section \ref{sec:applications} at the end of the paper.
The precise definition is the following:

\begin{definition}[Admissible manifolds]\label{def:admmanintro}
  Let $n\ge3$.
  We say that a smooth manifold $M^{n}$ is \emph{admissible}
  if its metric has the form
  $dr^{2}+h(r)^{2}d \omega^{2}_{\mathbb{S}^{n-1}}$
  and $h(r)$ satisfies:
  \begin{enumerate}
  [noitemsep,topsep=0pt,parsep=0pt,partopsep=0pt,
  label=\textit{(\roman*)}]
    \item $\exists h_{\infty}\ge0$ such
    that
    $H(r):=h^{\frac{1-n}{2}}(h^{\frac{n-1}{2}})''=
      h_{\infty}+O(r^{-2})$ for $r\gg 1$.
    \item $H^{(j)}(r)=O(r^{-1})$ and
    $(h^{-\frac12})^{(j)}=O(r^{-\frac12-j})$
    for $r\gg1$ and $1\le j\le[\frac{n-1}{2}]$.
    \item There exist $c,\delta_{0}>0$ such that
    for $r>0$ we have $h(r)\ge cr$ while the function
    $P(r)=rH(r)-rh_{\infty}+\frac{1-\delta_{0}}{4r}$
    satisfies the condition
    $P(r)\ge 0 \ge P'(r)$.
  \end{enumerate}
\end{definition}

Note that (i) is a form of asymptotic convexity,
while (iii) is effective
essentially on a bounded region. Condition (ii), on the other hand,
is weaker and excludes singularities of the metric at infinity.
The parameter $h_{\infty}$ can be understood as a measure of
the curvature of the manifold at infinity; $h_{\infty}=0$ means
essentially that the manifold is asymptotically flat, while the case
$h_{\infty}>0$ includes examples with large asymptotic curvature,
like the hyperbolic spaces.

Now assume both $M^{n}$ and $N^{\ell}$
are rotationally symmetric manifolds, with global metrics
\begin{equation}\label{eq:metrics}
  M^{n}:\quad
  dr^{2}+h(r)^{2}d \omega_{\mathbb{S}^{n-1}}^{2},
  \qquad
  N^{\ell}:\quad
  d \phi^{2}+g(\phi)^{2}d \chi_{\mathbb{S}^{\ell-1}}^{2}
\end{equation}
where $d \omega_{\mathbb{S}^{n-1}}^{2}$
and
$d \chi_{\mathbb{S}^{\ell-1}}^{2}$
are the standard metrics on the unit sphere. We recall the
\emph{equivariant ansatz}
(see \cite{ShatahStruwe98-a}):
writing the map $u=(\phi,\chi)$ in coordinates
on $N^{\ell}$,
the radial component $\phi=\phi(t,r)$ depends only on time
and $r$, the radial coordinate on $M^{n}$,
while the angular component
$\chi=\chi(\omega)$ depends only on the angular coordinate
$\omega$ on $M^{n}$. It follows that
$\chi:\mathbb{S}^{n-1}\to \mathbb{S}^{\ell-1}$ must be a harmonic
polynomial map of degree $k$, whose energy density is $k(k+n-2)$
for some integer $k\ge1$. On the other hand $\phi(t,r)$
must satisfy the \emph{$\bar{\ell}$-equivariant wave map equation}
\begin{equation}\label{eq:equivariant}
  \phi_{tt}-\phi_{rr}-(n-1)\frac{h'(r)}{h(r)}\phi_{r}
  +\frac{\bar{\ell}}{h(r)^{2}}g(\phi)g'(\phi)=0
\end{equation}
where $\bar{\ell}=k(k+n-2)$ and for which one considers the Cauchy problem with initial data
\begin{equation}\label{eq:equivdata}
  \phi(0,r)=\phi_{0}(r),
  \qquad
  \phi_{t}(0,r)=\phi_{1}(r).
\end{equation}
When $h(r)=r$ the base space is the flat $\mathbb{R}^{n}$ and
\eqref{eq:equivariant} reduces to the equation originally
studied in \cite{ShatahTahvildar-Zadeh94-a}.

In the following statement we use the
notation $|D_{M}|=(-\Delta_{M})^{\frac12}$,
where $\Delta_{M}$ is the Laplace-Beltrami operator
on $M^{n}$. If
$v:M^{n}\to N^{\ell}$ is an equivariant map of the form
$v=(\phi(r),\chi(\omega))$
with $\chi:\mathbb{S}^{n-1}\to \mathbb{S}^{\ell-1}$
a fixed harmonic map, its Sobolev $H^{s}(M^{n};N^{\ell})$
norm can be equivalently expresssed as
\begin{equation*}
  \|v\|_{H^{s}(M^{n};N^{\ell})}
  \simeq
  \|\phi\|_{H^{s}}:=
  \|(1-\Delta_{M})^{\frac s2}\phi\|_{L^{2}(M^{n})}.
\end{equation*}
We define also the weighted Sobolev space
$H^{s}_{q}(w)$ of radial functions on $M^{n}$
with norm
\begin{equation*}
  \|\phi\|_{H^{s}_{q}(w)}:=
  \|w^{-1}(|x|)\phi(|x|)\|_{H^{s}_{q}(\mathbb{R}^{n+2k})},
  \qquad
  w(r):=r^{k}\frac{r^{\frac{n-1}{2}}}{h(r)^{\frac{n-1}{2}}}.
\end{equation*}
and we choose the indices $(p,q)$ as
\begin{equation}\label{eq:pqind}
  p=\frac{4(m+1)}{m+3},
  \qquad
  q=\frac{4m(m+1)}{2m^{2}-m-5},
  \qquad
  m=n+2k.
\end{equation}
The notation $L^{\infty}H^{s}\cap CH^{s}$ denotes the space of
continuous bounded functions from $\mathbb{R}$ to
$H^{s}$, while $L^{p}H^{s}_{q}(w)$ is the space
of functions $\phi(t,r)$ which are $L^{p}$ in time
with values in $H^{s}_{q}(w)$.
Our main result is the following:

\begin{theorem}[Global existence in the critical norm]
  \label{the:main}
  Let $n\ge3$, $k\ge1$, $\bar{\ell}=k(k+n-2)$
  and $p,q$ as in \eqref{eq:pqind}.
  Assume $M^{n}$ and $N^{\ell}$ are two rotationally
  invariant manifolds with metrics given by \eqref{eq:metrics},
  with $M^{n}$ admissible, and let $h_{\infty}$ be the
  limit of $h^{\frac{1-n}{2}}(h^{\frac{n-1}{2}})''$ as
  $r\to \infty$. Consider the Cauchy problem
  \eqref{eq:equivariant}, \eqref{eq:equivdata}.

  If $h_{\infty}>0$ and
  $\| \phi_{0}\|_{H^{\frac{n}{2}}}+
    \|\phi_{1}\|_{H^{\frac{n}{2}-1}}$
  is sufficiently small, the problem
  has a unique global solution
  $\phi(t,r)\in L^{\infty}H^{\frac{n}{2}}
    \cap CH^{\frac{n}{2}}
    \cap L^{p}H^{\frac{n-1}{2}}_{q}(w)$.

  If $h_{\infty}=0$ and
  $\||D_{M}|^{\frac12} \phi_{0}\|_{H^{\frac{n-1}{2}}}+
    \||D_{M}|^{-\frac12}\phi_{1}\|_{H^{\frac{n-1}{2}}}$
  is sufficiently small, the problem
  has a unique global solution
  $\phi(t,r)$ with
  $|D_{M}|^{\frac12}\phi\in L^{\infty}H^{\frac{n-1}{2}}
    \cap CH^{\frac{n-1}{2}}$
  and
  $\phi\in L^{p}H^{\frac{n-1}{2}}_{q}(w)$.
\end{theorem}

\begin{remark}[Scattering]\label{rem:scatt}
  It is not difficult to prove that the solutions
  constructed in Theorem \ref{the:main} scatter to solutions
  of the \emph{linear} equivariant equation
  \begin{equation*}
    \phi_{tt}-\phi_{rr}-(n-1)\frac{h'(r)}{h(r)}\phi_{r}=0
  \end{equation*}
  in $H^{\frac n2}\times H^{\frac n2-1}$
  as $t\to\pm \infty$, by standard arguments; we omit the details.
\end{remark}

\begin{remark}[Local existence with large data]\label{rem:LWP}
  By a simple modification in the proof
  one can show that the small data assumption can be replaced by the
  weaker assumption that the linear part of the flow is
  sufficiently small. This in particular implies existence and
  uniqueness of a time local solution for large data in the
  same regularity class
  (see Remark \ref{rem:LWPlargedata} for a sketch of the proof).
\end{remark}

Thus global existence of small equivariant wave maps
on admissible manifolds holds in the critical space
$H^{\frac n2}\times H^{\frac n2-1}$, as in the case of a flat
base manifold.
The solution enjoys additional
$L^{p}L^{q}$ integrability properties, determined
by the Strichartz estimates used in the proof.
This has the usual drawback that uniqueness holds only
in a restricted space.
\emph{Unconditional uniqueness}
in the critical space
without additional restrictions
was proved recently for general wave maps on Minkowski space
in \cite{MasmoudiPlanchon12-a}.
We conjecture that a similar result holds also in
our situation; as a
partial workaround, we prove that if the
regularity of the initial data is increased by
$\delta=\frac{1}{m+1}$ then uniqueness holds in
the space $CH^{\frac n2+\frac{1}{m+1}}$:

\begin{theorem}[Higher regularity and unconditional uniqueness]
  \label{the:uuu}
  Consider \eqref{eq:equivariant}, \eqref{eq:equivdata}
  under the assumptions of Theorem \ref{the:main},
  and let $0\le \delta<k$.

  If $h_{\infty}>0$ and
  $\| \phi_{0}\|_{H^{\frac{n}{2}+\delta}}+
    \|\phi_{1}\|_{H^{\frac{n}{2}-1+\delta}}$
  is sufficiently small, the problem
  has a unique global solution
  $\phi\in L^{\infty}H^{\frac{n}{2}+\delta}
    \cap CH^{\frac{n}{2}+\delta}
    \cap L^{p}H^{\frac{n-1}{2}+\delta}_{q}(w)$.
  Moreover, if $\delta\ge \frac{1}{m+1}$, this is
  the unique solution in $CH^{\frac n2+\delta}$.

  If $h_{\infty}=0$ and
  $\||D_{M}|^{\frac12} \phi_{0}\|_{H^{\frac{n-1}{2}+\delta}(M)}+
    \||D_{M}|^{-\frac12}\phi_{1}\|_{H^{\frac{n-1}{2}+\delta}(M)}$
  is sufficiently small, Problem
  \eqref{eq:equivariant}, \eqref{eq:equivdata}\
  has a unique global solution
  $\phi$ with
  $|D_{M}|^{\frac12}\phi\in L^{\infty}H^{\frac{n-1}{2}+\delta}(M)
    \cap CH^{\frac{n-1}{2}+\delta}(M)$
  and
  $\phi\in L^{p}H^{\frac{n-1}{2}+\delta}_{q}(w)$.
  Moreover, if
  $\delta\ge \frac{1}{m+1}$, this is the unique solution
  with $|D_{M}|^{\frac12}\phi\in CH^{\frac{n-1}{2}+\delta}$.
\end{theorem}

\begin{remark}[Examples of admissible manifolds]\label{rem:admman}
  In Section \ref{sec:applications} we discuss at some length the
  admissibility assumption. In particular we prove
  that suitable perturbations of
  admissible manifolds are also admissible; this allows to
  substantially enlarge the list of explicit examples.
  The following manifolds are included in the class:
  \begin{itemize}
    \item The euclidean $\mathbb{R}^{n}$ and,
    more generally, rotationally invariant,
    asymptotically flat spaces of dimension
    $n\ge3,$. The precise condition
    is the following: the radial component of
    the metric has the form $h_{\epsilon}(r)=r+\mu(r)$,
    with $\mu:\mathbb{R}^{+}\to \mathbb{R}$ such that
    for small $\epsilon>0$
    \begin{equation*}
      |\mu(r)|+r|\mu'(r)|+r^{2}|\mu''(r)|
      +r^{3}|\mu'''(r)|\le \epsilon r
      \quad\text{for all}\quad
      r>0
    \end{equation*}
    and
    \begin{equation*}
      \textstyle
      |\mu^{(j)}(r)|\lesssim r^{1-j}
      \quad\text{for}\quad
      r\gg1,
      \quad
      j\le[\frac{n-1}{2}]+2.
    \end{equation*}
    \item Real hyperbolic spaces $\mathbb{H}^{n}$
    with $n\ge3$, for which
    $h(r)=\sinh r$; more generally,
    rotationally invariant perturbations of $\mathbb{H}^{n}$
    with a metric $h_{\epsilon}(r)=\sinh r+\mu(r)$,
    with $\mu:\mathbb{R}^{+}\to \mathbb{R}$ such that
    for small $\epsilon>0$
    \begin{equation*}
      |\mu(r)|+|\mu'(r)|+|\mu''(r)|+|\mu'''(r)|\le
      \epsilon \bra{r}^{-3}\sinh r
      \quad\text{for all}\quad
      r>0
    \end{equation*}
    and
    \begin{equation*}
      |\mu^{(j)}(r)|\lesssim r^{-1}e^{r}
      \quad\text{for}\quad
      r\gg 1,
      \quad
      j\le[\frac{n-1}{2}]+2.
    \end{equation*}
    \item Some classes of rotationally invariant
    manifolds with a metric $h(r)$ of polynomial
    growth $h(r)\sim r^{M}$, wher $M$ can be any $M\ge1$.
  \end{itemize}
\end{remark}

\begin{remark}[Strichartz estimates]\label{rem:}
  The crucial tools in Theorems \ref{the:main},
  \ref{the:uuu} are smoothing and
  Strichartz estimates for wave equations
  defined on admissible manifolds, which are proved in
  Section \ref{sec:strichartz}. To our knowledge,
  Strichartz estimates on curved backgrounds were essentially
  known only for asymptotically flat manifolds,
  see e.g.
  \cite{StaffilaniTataru02-a},
  \cite{RobbianoZuily05-a},
  \cite{BoucletTzvetkov08-a}
  \cite{RodnianskiTao11-a},
  among the others.
  For the case of manifolds with a nonvanishing curvature
  at infinity, such estimates are known in the special
  case of real hyperbolic spaces, see
  \cite{AnkerPierfelice09-a},
  \cite{AnkerPierfeliceVallarino12-a}.
\end{remark}

\begin{remark}[The case $n=2$]\label{rem:2D}
  Our proof does not cover the case $n=2$. Indeed,
  the main obstruction is the smoothing estimate of
  Theorem \ref{the:smoo}, which fails when $n=2$, and
  is the crucial step in the proof of Strichartz estimates.
  It is indeed possible to prove a smoothing estimate also in the
  low dimensional case, but this requires a substantial modification
  in the argument (in particular, it is necessary to use time
  dependent Morawetz multipliers). We plan to address
  this problem in a further work.
\end{remark}

The plan of the paper is the following.
In Section \ref{sec:chvar} we transform the equivariant wave
map equation in an equation with potential defined on
$\mathbb{R}^{m}$, with $m=n+2k\ge5$. Since we need Strichartz
estimates at the level of $H^{\frac n2}$ regularity, we develop
some tools to commute derivatives with the flow, and the
lemmas to this purpose are proved in the same section.
In Section \ref{sec:strichartz} we prove smoothing and
Strichartz estimates for the transformed linear equation;
they hold under suitable assumptions on the potential,
which translate into the definition of admissible
manifold.
Section \ref{sec:fixedpoint} is devoted to
the proof of global existence with small data for a
radial nonlinear wave equation with critical nonlinearity
depending also on the radial variable.
Finally, in Section \ref{sec:applications}
we apply the result to the original wave map equation
and we discuss the definition of admissible manifold
in detail.

\subsection*{Acknowledgments}\label{sub:ack}
We would like to thank the Referees for the very useful suggestions
which helped to improve the paper.

\section{Reduction to a perturbed problem and equivalence
of norms}\label{sec:chvar}

The component $g(s)$ in the metric of the target manifold
is the restriction of an odd, smooth function on $\mathbb{R}$.
Thus we can write
\begin{equation*}
  \bar{\ell} g(s)g'(s)=\bar{\ell} s+s^{3}\Gamma(s)
  % ,\qquad \Gamma(0)\neq0
\end{equation*}
with $\Gamma(s)$ smooth. Applying the change of variable
\begin{equation*}
  \phi(t,r)=\psi(t,r)\cdot w(r),
  \qquad
  w(r):=
  \frac{r^{k+\frac{n-1}{2}}}{h(r)^{\frac{n-1}{2}}},
\end{equation*}
equation \eqref{eq:equivariant} reduces to
\begin{equation}\label{eq:perteq}
  \psi_{tt}-\psi_{rr}-\frac{m-1}{r}\psi_{r}+V(r)\psi+
  \frac{r^{m-1}}{h(r)^{n+1}}\psi^{3}
  \Gamma\left(\frac{r^{\frac{m-1}{2}}}{h^{\frac{n-1}{2}}}
  \psi\right)=0
\end{equation}
where $m=2k+n$ and
\begin{equation*}
  V(r)=
  \frac{n-1}{2}
  \left[
    \frac{h''}{h}
    +\frac{n-3}{2}\left(
      \frac{h'^{2}}{h^{2}}-\frac{1}{r^{2}}
    \right)
  \right]
  +k(k+n-2)\left(
    \frac{1}{h^{2}}-\frac{1}{r^{2}}
  \right).
\end{equation*}
Note that the function $h(r)$ can be extended to a smooth
odd function on $\mathbb{R}$ and satisfies $h(0)=h''(0)=0$,
$h'(0)=1$. As a consequence,
the potential $V(r)$ is smooth at the origin,
and has a critical decay $\sim r^{-2}$
in general. Our main goal will be now to prove
Strichartz estimates for the transformed equation
\eqref{eq:perteq}, and this will be the object of the
next section. Note that we require estimates at the level
of the $H^{\frac n2}$ norm, thus we need to be able to commute
$\frac n2$ derivatives with functions of the operator
$-\Delta+V$; the rest of the section is devoted to
the necessary tools for this step.

In following, $c(x)$ is a measurable
real valued function on $\mathbb{R}^{m}$ and the operator
\begin{equation*}
  L=-\Delta+c(x)
\end{equation*}
is a selfadjoint Schr\"{o}dinger
operator on $L^{2}(\mathbb{R}^{m})$.
We shall make different sets of assumptions
on the potential $c(x)$, but in all cases they
imply that $-\Delta+c(x)$ has a unique
selfadjoint extension by well known results.
This fact will be tacitly used, and in particular we shall
use the spectral calculus associated to the operator
$-\Delta+c(x)$ without further notice.

The first result is
contained in \cite{BurqPlanchonStalker04-a} but we include
a short proof for completeness:

\begin{lemma}[]\label{lem:sobhom}
  Let $m\ge3$ and assume
  \begin{equation}\label{eq:assc}
    \textstyle
    \frac{C}{|x|^{2}}\ge c(x)\ge-\frac{(m-2)^{2}-\delta}{4|x|^{2}}
  \end{equation}
  for some $C,\delta>0$. Then $-\Delta+c$ is a positive operator,
  and for all $|s|\le1$ we have the equivalences
  \begin{equation}\label{eq:equivsovhom}
    \|(-\Delta+c)^{\frac s2}u\|_{L^{2}}
    \simeq
    \|u\|_{\dot H^{s}},
    \qquad
    \|(1-\Delta+c)^{\frac s2}u\|_{L^{2}}
    \simeq
    \|u\|_{H^{s}}.
  \end{equation}
\end{lemma}

\begin{proof}%[of ...]
  By Hardy's inequality we have
  $\|\nabla u\|^{2}_{L^{2}}
  \gtrsim(Lu,u) \gtrsim\|\nabla u\|^{2}_{L^{2}}$
  and this implies
  the case $s=1$. The case $s=-1$ is obtained by duality,
  and the remaining cases follow by interpolation.
  The proof of the second equivalence is almost identical.
\end{proof}

When the potential is smoother, we have a more general result
for higher order nonhomogeneous norms. The following estimate is
not sharp but sufficient for our purposes.
We use the notation $\lceil s\rceil$ for the least integer
$\ge s$.

\begin{lemma}\label{sob1}
  Let $s\ge0$,  $1<p<\infty$ and
  assume $c(x)$ has bounded derivatives up to the order
  $2\lceil s\rceil-2$.
  Then there exists constants
  $K_{0},C$ depending on $s,p$ and on the potential $c(x)$
  such that, for all $K\ge K_{0}$,
  \begin{equation}\label{eq1-sob1}
     C^{-1}
     \|(K-\Delta)^{s} u\|_{L^p} \le
     \|(K-\Delta+c)^s u\|_{L^p} \le
     C\|(K-\Delta)^{s} u\|_{L^p}.
  \end{equation}
\end{lemma}

\begin{proof}
  In the course of the proof, the letter $C$
  denotes several constants which are independent of $K$.
  Consider first the case $s=k>0$ is an integer.
  We can write
  \begin{equation}\label{eq2-sob1}
    \textstyle
    (K-\Delta+c)^ku=(K-\Delta)^k u+\sum\nolimits
    C\cdot K^{h} \partial^{\alpha}(c^{\ell})
    (\partial^\beta u),
  \end{equation}
  where the sum extends over all multiindices $\alpha,\beta$
  and integers $h,\ell$ such that
  \begin{equation*}
    \frac12(
      |\alpha|+|\beta|)+h+\ell=k,
    \qquad
    \ell\ge1.
  \end{equation*}
  % Without loss of generality, we can assume that
  % $\alpha_{\ell}$ is the multiindex of largest length
  % among $\alpha_1,\dots,\alpha_\ell$, and this implies
  % $|\alpha_i| \le k-1, \; i=1,\dots, \ell-1$.
  Note that $|\alpha| \le 2k-2$, so that
  $\|\partial^{\alpha}(c^{\ell}) \|_{L^{\infty}}\le
    C\|c\|_{W^{2k-2,\infty}}^{\ell}$.
  If we take
  \begin{equation*}
    K_{1}:= \|c\|_{W^{2k-2,\infty}}
  \end{equation*}
  we can estimate the $L^p$-norm of the generic term
  of the sum in \eqref{eq2-sob1} as follows
  \begin{equation*}
    CK^{h}
    \|\partial^{\alpha}(c^{\ell})
      \partial^\beta u\|_{L^{p}}
    \le
    CK^{h}K_{1}^{\ell}
    \|\partial^{\beta}u\|_{L^p}=
    C(K_{1}/K)^{\ell}K^{h+\ell}
    \|\partial^{\beta}u\|_{L^p}.
  \end{equation*}
  By the Mikhlin multiplier theorem and
  recalling that $|\beta|+2h+2\ell\le 2k$, this can be
  estimated with
  $$ \le C (K_{1}/K)^{\ell} \| (K-\Delta)^{k} u \|_{L^p}.$$
  Thus if we take $K\gg K_{1}=\|c\|_{W^{2k-2,\infty}}$
  we obtain from \eqref{eq2-sob1} that
  \begin{equation*}
    \textstyle
    (K-\Delta+c)^ku=(K-\Delta)^k u+I
    \quad\text{where}\quad
    \|I\|_{L^{p}}\le \epsilon
    \|(K-\Delta)^k u\|_{L^{p}},
    \quad
    \epsilon<1
  \end{equation*}
  and this cocnludes the proof in the case $s=k$.

  If $s$ is not an integer, we consider the analytic
  family of operators $T_{z}=(K-\Delta+c)^{z}(K-\Delta)^{-z}$
  and we apply Stein interpolation. In this step, the
  $L^{p}$ boundedness of the operators
  $(K-\Delta+c)^{iy}$ for $y\in \mathbb{R}$ follows
  e.g.~from the general method of
  \cite{CacciafestaDAncona12-a}
  since the heat kernel $e^{-t(K-\Delta+c)}$ satisfies
  an upper gaussian estimate (note that $c+K\ge0$).
\end{proof}

From Mikhlin-H\"{o}rmander we know that
\begin{equation}\label{eq:MH}
  \|(1-\Delta)^{s}v\|_{L^{p}}\simeq
  \|(K-\Delta)^{s}v\|_{L^{p}}
\end{equation}
for all $K>0$ and $1<p<\infty$, with a constant depending on
$K$. A similar property holds for the operators
$(K-\Delta+c)^{s}$, at least in the case $p=2$ and
if the operator is positive:

\begin{lemma}[]\label{lem:irrelevantM}
  Let $m\ge3$, $s\in \mathbb{R}$ and
  assume $c(x)$ satisfies \eqref{eq:assc}. Then for all $K>0$
  we have the equivalence
  \begin{equation}\label{eq:equivbrac}
    \|(K-\Delta+c)^{s}v\|_{L^{2}(\mathbb{R}^{m})}
    \simeq
    \|(1-\Delta+c)^{s}v\|_{L^{2}(\mathbb{R}^{m})}.
  \end{equation}
\end{lemma}

\begin{proof}%[of ...]
  By complex interpolation (which does not require a
  gaussian estimate since we are in the elementary $L^{2}$
  case) it is sufficient to prove the
  equivalence when $s$ is a half integer; and since
  $(K-\Delta+c)^{\frac12}$ and $(1-\Delta+c)^{\frac12}$
  commute, it is
  sufficient to prove it for $s=\frac12$.
  But in this case the equivalence is obvious since
  \begin{equation*}
    ((K-\Delta+c)v,v)\simeq
    K\|v\|_{L^{2}}^{2}+\|\nabla v\|_{L^{2}}^{2}
  \end{equation*}
  by Lemma \ref{lem:sobhom}.
\end{proof}

\begin{lemma}[]\label{sob2}
  Let $m\ge3$, $s\ge0$, and assume $c(x)$ satisfies
  for some $C,\delta,C_{0}>0$ condition
  \eqref{eq:assc} and
  \begin{equation}\label{eq:hardyassc}
    |\partial^{\alpha}c(x)|\le C_{0}\bra{x}^{-1},
    \qquad
    |\alpha|\le\lfloor s\rfloor+2.
  \end{equation}
  Then, using the notations
  $|D_{c}|=(-\Delta+c)^{\frac12}$ and
  $\bra{D_{c}}=(1-\Delta+c)^{\frac12}$,
  the following equivalences hold:
  \begin{equation}\label{eq1-lem2}
    \||D|^\frac12\langle D \rangle ^{s} u\|_{L^2}
    \simeq
    \||D_c|^\frac12\bra{D_{c}}^{s} u\|_{L^2}
    \simeq
    \||D|^\frac12\bra{D_{c}}^{s} u\|_{L^2}
  \end{equation}
  and
  \begin{equation}\label{eq1b-lem2}
    \||D|^{-\frac12}\langle D \rangle ^{s} u\|_{L^2}
    \simeq
    \||D_c|^{-\frac12}\bra{D_{c}}^{s} u\|_{L^2}
    \simeq
    \||D|^{-\frac12}\bra{D_{c}}^{s} u\|_{L^2}
  \end{equation}
  with implicit constants depending on $c(x)$ and $s$.
\end{lemma}

\begin{proof}%[of ...]
  Note that it is sufficient to prove \eqref{eq1-lem2}
  with $(K-\Delta)^{\frac12}$ and $(K-\Delta+c)^{\frac12}$
  in place of $\bra{D}$ and $\bra{D_{c}}$
  respectively, with an arbitrarily large $K$, thanks to
  the equivalence \eqref{eq:equivbrac}. In the following
  we shall use the same notation for all values of $K$.

  The claim is that for a fixed $s\ge0$ and $z=\frac12$
  the operators
  \begin{equation*}
    |D|^{z}\bra{D}^{s}\bra{D_{c}}^{-s}|D_{c}|^{-z},
    \qquad
    |D_{c}|^{z}\bra{D_{c}}^{s}\bra{D}^{-s}|D|^{-z}
  \end{equation*}
  are bounded on $L^{2}$. By interpolation it is sufficient
  to prove the case $z=1$. Thus, if we define the analytic
  families of operators
  \begin{equation*}
    T_{z}=|D|\bra{D}^{z}\bra{D_{c}}^{-z}|D_{c}|^{-1},
    \qquad
    S_{z}=|D_{c}|\bra{D_{c}}^{z}\bra{D}^{-z}|D|^{-1}
  \end{equation*}
  we have to prove that $T_{z},S_{z}$
  are $L^{2}$ bounded when $z=s$. Again by interpolation,
  it is sufficient to prove that $T_{z}$ and $S_{z}$
  are bounded when $z=2k$ is an even integer.
  Thus we are reduced to the estimate
  \begin{equation*}
    \||D| (K-\Delta)^{k}v\|_{L^{2}}\simeq
    \||D_{c}| (K-\Delta+c)^{k}v\|_{L^{2}}.
  \end{equation*}
  Using Lemma \ref{lem:sobhom} we can replace $|D_{c}|$
  with $|D|$ at the r.h.s., and the claim is implied by
  \begin{equation}\label{eq:claim}
    \|\nabla (K-\Delta)^{k}v\|_{L^{2}}\simeq
    \|\nabla (K-\Delta+c)^{k}v\|_{L^{2}}.
  \end{equation}
  We prove \eqref{eq:claim} by induction on $k$, using
  the equivalence
  \begin{equation*}
    \textstyle
    \|(K-\Delta)^{k}v\|_{L^{2}}
    \simeq
    K^{k}\|v\|_{L^{2}}+\|v\|_{\dot H^{2k}}
    \simeq
    \sum_{j=0}^{2k}K^{k-j/2}\|v\|_{\dot H^{j}}
  \end{equation*}
  (implicit constants independent of $K$). By the
  induction step $k\to k+1$ we obtain
  \begin{equation}\label{eq:claim2}
  \begin{split}
    \|\nabla (K-\Delta+c)^{k+1}v\|_{L^{2}}
    \simeq &
    \|(K-\Delta)^{k}\nabla[(K-\Delta+c)v]\|_{L^{2}}
    =
    \\
    =&\|(K-\Delta)^{k+1}\nabla v+
        (K-\Delta)^{k}\nabla(cv)\|_{L^{2}}.
  \end{split}
  \end{equation}
  We have
  \begin{equation*}
    \textstyle
    \|(K-\Delta)^{k}\nabla(cv)\|_{L^{2}}
    \simeq
    K^{k}\|cv\|_{\dot H^{1}}+
    \|cv\|_{\dot H^{2k+1}}
  \end{equation*}
  and also, using assumption \eqref{eq:hardyassc},
  \begin{equation*}
    \|cv\|_{\dot H^{j}}
    \lesssim
    C_{0}\||x|^{-1}v\|_{L^{2}}+
    C_{0}\|\nabla v\|_{H^{j-1}}
    \simeq
    C_{0}\|\nabla v\|_{H^{j-1}}
  \end{equation*}
  by Hardy's inequality, so that
  \begin{equation*}
    \textstyle
    \|(K-\Delta)^{k}\nabla(cv)\|_{L^{2}}
    \lesssim
    K^{k}C_{0}\|\nabla v\|_{L^{2}}+
    C_{0}\|\nabla v\|_{H^{2k}}
    \lesssim
    \frac{C_{0}}{K}
    \|(K-\Delta)^{k+1}\nabla v\|_{L^{2}}.
  \end{equation*}
  Taking $K\gg C_{0}$, we see that we can absorb the last
  term in \eqref{eq:claim2} and we obtain \eqref{eq:claim}.

  The proof of \eqref{eq1b-lem2} is analogous.
  Instead of \eqref{eq:claim} we arrive at
  \begin{equation*}
    \||D|^{-1}(K-\Delta)^{k}v\|_{L^{2}}
    \simeq
    \||D|^{-1}(K-\Delta+c)^{k}v\|_{L^{2}}.
  \end{equation*}
  As before, the induction step gives
  \begin{equation}\label{eq:claim2bis}
    \||D|^{-1} (K-\Delta+c)^{k+1}v\|_{L^{2}}
    \simeq
    \|(K-\Delta)^{k+1}|D|^{-1} v+
        (K-\Delta)^{k}|D|^{-1}(cv)\|_{L^{2}}
  \end{equation}
  and we must absorb the last term by taking $K$ sufficiently
  large. We can write
  \begin{equation*}
    \textstyle
    \|(K-\Delta)^{k}|D|^{-1}(cv)\|_{L^{2}}
    \simeq
    K^{k}\||D|^{-1}(cv)\|_{L^{2}}+
    \|cv\|_{\dot H^{2k-1}}
  \end{equation*}
  and this must be controlled by the main term which is
  \begin{equation*}
    \textstyle
    \|(K-\Delta)^{k+1}|D|^{-1}v\|_{L^{2}}
    \simeq
    K^{k+1}\||D|^{-1}v\|_{L^{2}}+\|v\|_{\dot H^{2k+1}}
    \simeq
    \sum_{j=0}^{2k+2}K^{k+1-j/2}\|v\|_{\dot H^{j-1}}
  \end{equation*}
  Hence, using assumption \eqref{eq:hardyassc}, we have
  by Hardy's inequality
  \begin{equation*}
    \|cv\|_{\dot H^{2k-1}}
    \lesssim
    C_{0}\|v\|_{H^{2k}}
    \lesssim
    C_{0}K^{-\frac12}
    \|(K-\Delta)^{k+1}|D|^{-1}v\|_{L^{2}};
  \end{equation*}
  in a similar way we have, also by
  Hardy's inequality,
  \begin{equation*}
    \||D|c(x)|D|^{-1}v\|_{L^{2}}
    \simeq
    \|\nabla(c(x)|D|^{-1}v)\|_{L^{2}}
    \lesssim
    C_{0}\|v\|_{L^{2}}
  \end{equation*}
  and this estimate by duality is equivalent to
  \begin{equation*}
    \||D|^{-1}(cv)\|_{L^{2}}
    \lesssim
    C_{0}\||D|^{-1}v\|_{L^{2}}.
  \end{equation*}
  Summing up, we obtain
  \begin{equation*}
    \|(K-\Delta)^{k}|D|^{-1}(cv)\|_{L^{2}}
    \lesssim
    C_{0}(K^{-1}+K^{-\frac12})
    \|(K-\Delta)^{k+1}|D|^{-1}v\|_{L^{2}}
  \end{equation*}
  and taking $K$ sufficiently large in \eqref{eq:claim2bis}
  we conclude the proof.
\end{proof}

We shall also need a lemma relating Sobolev norms on
spaces with different dimension for radial functions,
which extends Lemma 1.3 in
\cite{ShatahTahvildar-Zadeh94-a}.

\begin{lemma}[]\label{lem:dimchange}
  Let $n\ge2$, $k\ge1$ be integers and $s\ge0$.
  Then for all radial functions $v(r)$ we have the
  equivalence of norms
  \begin{equation}\label{eq:dimch}
    \||x|^{k}v(|x|)\|_{\dot H^{s}(\mathbb{R}^{n}_{x})}
    \simeq
    \|v(|y|)\|_{\dot H^{s}(\mathbb{R}^{n+2k}_{y})}
  \end{equation}
  with implicit constants depending only on
  $s,n,k$.
\end{lemma}

\begin{proof}%[of ...]
  The following pointwise equivalence is valid
  for all radial functions $v(r)$ and integers $N\ge0$:
  \begin{equation}\label{eq:pteq}
    \sum_{|\alpha|=N}|D^{\alpha}v(|x|)|
    \simeq
    |\partial_{r}^{N}v(|x|)|,
  \end{equation}
  where $\partial_{r}=\widehat{x}\cdot \nabla_{x}$ denotes the
  radial derivative and $\widehat{x}=\frac{x}{|x|}$;
  the implicit constants in \eqref{eq:pteq}
  depend on $N,n$ but not on $v$ or $x$.
  We prove \eqref{eq:pteq} by induction on $N$. For $N=1$
  it follows immediately from
  $\nabla v(|x|)=v'(|x|)\nabla|x|=v'(|x|)\widehat{x}$.
  Now assume the equivalence holds for some $N$; then we can write
  \begin{equation*}
    \sum_{|\alpha|=N+1}|D^{\alpha}v|
    \simeq
    \sum_{\ell=1}^{n}
    \sum_{|\beta|=N}|D^{\beta}\partial_{\ell}v|
    \simeq
    \sum_{\ell=1}^{n}|\partial_{r}^{N}\partial_{\ell}v|.
  \end{equation*}
  Since $\partial_{\ell}v=\widehat{x}_{\ell}v'$ and
  $\partial_{r}\widehat{x}_{\ell}=0$, we have
  $\partial_{r}\partial_{\ell}v=\partial_{\ell}\partial_{r} v$
  and this implies
  \begin{equation*}
    \simeq
    \sum_{\ell=1}^{n}|\partial_{\ell}\partial_{r}^{N}v|
    \simeq
    |\partial_{r}^{N+1}v|
  \end{equation*}
  which proves \eqref{eq:pteq}.

  In order to prove \eqref{eq:dimch}, we note that the case of
  general $k$ follows by repeated application of the case
  $k=1$; moreover, if \eqref{eq:dimch} is true for some $s=s_{0}$,
  by complex interpolation with the trivial case $s=0$, it
  holds for all $0\le s\le s_{0}$. In conclusion, it is sufficient
  to prove \eqref{eq:dimch} when $k=1$ and $s=N$ is an integer
  which we can assume large, say $N>n$. In this case we have,
  using \eqref{eq:pteq},
  \begin{equation*}
    \textstyle
    \||x|v(|x|)\|^{2}_{\dot H^{N}(\mathbb{R}^{n})}
    \simeq
    \|\partial_{r}^{N}(rv(r))\|^{2}_{L^{2}(\mathbb{R}^{n})}
    \simeq
    \int_{0}^{+\infty}r^{n-1}
    |r\partial_{r}^{N}v+N\partial_{r}^{N-1}v|^{2}dr
  \end{equation*}
  and to prove the claim it remains to check the
  equivalence
  \begin{equation}\label{eq:concl}
    \textstyle
    \int_{0}^{+\infty}r^{n-1}
    |r\partial_{r}^{N}v+N\partial_{r}^{N-1}v|^{2}dr
    \simeq
    \int_{0}^{+\infty}r^{n+1}
    |\partial_{r}^{N}v|^{2}dr.
  \end{equation}
  One side of \eqref{eq:concl} follows by
  the Cauchy-Schwartz and then Hardy's inequality:
  \begin{equation*}
    \textstyle
    \int_{0}^{+\infty}r^{n+1}
    |r^{-1}\partial_{r}^{N-1}v|^{2}dr
    \simeq
    \||x|^{-1}\partial_{r}^{N-1}v\|_{L^{2}(\mathbb{R}^{n+2})}^{2}
    \lesssim
    \|\nabla\partial_{r}^{N-1}v\|_{L^{2}(\mathbb{R}^{n+2})}^{2}
    \simeq
    \|\partial_{r}^{N}v\|_{L^{2}(\mathbb{R}^{n+2})}^{2}.
  \end{equation*}
  To prove the converse inequality, we expand the square
  at the left hand side of \eqref{eq:concl}:
  \begin{equation*}
    \textstyle
    \int_{0}^{+\infty}r^{n+1}
    |\partial_{r}^{N}v|^{2}dr
    +
    \int_{0}^{+\infty}r^{n-1}
    [
    N^{2}|\partial_{r}^{N-1}v|^{2}+
    2Nr\Re(\partial_{r}^{N-1}v \overline{\partial_{r}^{N}v})
    ]dr,
  \end{equation*}
  then we integrate by parts the last term
  \begin{equation*}
    \textstyle
    2N \int_{0}^{+\infty}r^{n}
    \Re(\partial_{r}^{N-1}v \overline{\partial_{r}^{N}v})
    dr=
    N \int_{0}^{+\infty}r^{n}
    \partial_{r}|\partial_{r}^{N-1}v|^{2}dr=
    -Nn\int_{0}^{+\infty}r^{n-1}
    |\partial_{r}^{N-1}v|^{2}dr
  \end{equation*}
  and in conclusion we obtain that the left hand side
  of \eqref{eq:concl} is equal to
  \begin{equation*}
    \textstyle
    \int_{0}^{+\infty}r^{n+1}
    |\partial_{r}^{N}v|^{2}dr
    +
    N(N-n)
    \int_{0}^{+\infty}r^{n-1}
    |\partial_{r}^{N-1}v|^{2}dr
    \ge
    \int_{0}^{+\infty}r^{n+1}
    |\partial_{r}^{N}v|^{2}dr
  \end{equation*}
  since $N>n$, which proves the claim.
\end{proof}

We finally consider a different type of equivalence, which is
related to the change of variable
\begin{equation}\label{eq:chov}
  \phi(r)=w(r)\cdot \psi(r),
  \qquad
  w(r):=\frac{r^{\frac{m-1}{2}}}{h(r)^{\frac{n-1}{2}}}=
        w_{0}(r)\cdot r^{\frac{m-n}{2}},
  \qquad
  w_{0}:=\left(\frac r h\right)^{\frac{n-1}{2}}.
\end{equation}
Note that $w_{0}(0)=1$.
The Laplace-Beltrami operator $\Delta_{M}$ on $M^{n}$,
with metric $dr^{2}+h(r)^{2}d \omega^{2}_{\mathbb{S}^{n-1}}$,
and the flat Laplacians $\Delta_{n}$ on $\mathbb{R}^{n}$
and $\Delta_{m}$ on $\mathbb{R}^{m}$
act on radial functions respectively as
\begin{equation}\label{eq:lapls}
  \Delta_{M}\phi=\phi''+(n-1)\frac{h'}{h}\phi',
  \qquad
  \Delta_{n} \psi=\psi''+\frac{n-1}{r}\psi',
  \qquad
  \Delta_{m} \psi=\psi''+\frac{m-1}{r}\psi'.
\end{equation}
The operators $\Delta_{M}$
and $\Delta_{n}$ are connected by the formula
\begin{equation}\label{eq:potV0}
  \textstyle
  w_{0}^{-1}\Delta_{M}w_{0}=
  \Delta_{n}-V_{0}(r),
  \qquad
  V_{0}:=
  \frac{n-1}{2}
  \left[
    \frac{h''}{h}
    +\frac{n-3}{2}\left(
      \frac{h'^{2}}{h^{2}}-\frac{1}{r^{2}}
    \right)
  \right].
\end{equation}

% \begin{lemma}[]\label{lem:equivLBnorm}
%   Let $m,n\ge2$, $s\ge0$,
%   % $1<p<\infty$,
%   $M^{n}$ a smooth rotationally symmetric
%   manifold with metric
%   $dr^{2}+h(r)^{2}d \omega_{\mathbb{S}^{n-1}}$, and let
%   $w(r)=r^{\frac{m-1}{2}}h(r)^{\frac{1-n}{2}}$ and $V_{0}(r)$
%   as in \eqref{eq:potV0}.
%   Assume that $V_{0}(|x|)$ has
%   bounded derivatives on $\mathbb{R}^{m}$ up to the order
%   $2\lceil s\rceil-2$.
%   Then $\exists K_{0}=K_{0}(h,m,n,s)\ge0$, such that
%   for all smooth radial functions $\phi(r)$
%   and all $K\ge K_{0}$ we have the equivalence
%   \begin{equation}\label{eq:equivLB}
%     \|(K-\Delta_{M})^{s}\phi\|_{L^{2}(M^{n})}.
%     \simeq
%     \|(K-\Delta_{m})^{s}(w^{-1}\phi)\|_{L^{2}(\mathbb{R}^{m})}
%   \end{equation}
% \end{lemma}

\begin{lemma}[]\label{lem:equivLBnorm2}
  Let $m>n\ge3$, $s\ge0$,
  % $1<p<\infty$,
  $M^{n}$ a smooth rotationally symmetric
  manifold with metric
  $dr^{2}+h(r)^{2}d \omega^{2}_{\mathbb{S}^{n-1}}$, and let
  $w(r)=r^{\frac{m-1}{2}}h(r)^{\frac{1-n}{2}}$ and $V_{0}(r)$
  as in \eqref{eq:potV0}.
  Assume that $V_{0}(|x|)$ has
  bounded derivatives on $\mathbb{R}^{m}$ up to the order
  $2\lceil s\rceil-2$, that
  $\partial^{\alpha}V_{0}=O(|x|^{-1})$ as
  $|x|\to \infty$ for $|\alpha|\le[s]$, and that
  $V_{0}$ satisfies condition \eqref{eq:assc}.
  Then for all smooth functions $\phi(r)$ on $M^{n}$
  which depend only on the radial coordinate,
  we have the equivalence
  \begin{equation}\label{eq:equivLB2}
    % \|w^{\frac2p-1}(K-\Delta_{M})^{s}\phi\|_{L^{p}(M^{n})}.
    \|\phi\|_{H^{s}(M^{n})}
    \simeq
    \|w(|y|)^{-1}\phi(|y|)\|_{H^{s}(\mathbb{R}^{m}_{y})}
  \end{equation}
  Moreover, using the notations
  $|D_{M}|=(-\Delta_{M})^{\frac12}$,
  $\bra{D_{M}}=(1-\Delta_{M})^{\frac12}$,
  we have
  \begin{equation}\label{eq:equivLB2bis}
    % \|w^{\frac2p-1}(K-\Delta_{M})^{s}\phi\|_{L^{p}(M^{n})}.
    \||D_{M}|^{\pm \frac12}\bra{D_{M}}^{s}\phi\|_{L^{2}(M^{n})}
    \simeq
    \||D|^{\pm\frac12}\bra{D}^{s}[w(|y|)^{-1}\phi(|y|)]\|
         _{L^{2}(\mathbb{R}^{m}_{y})}
  \end{equation}
\end{lemma}

\begin{proof}%[of ...]
  From the first formula in \eqref{eq:potV0} we obtain,
  for all $s,K\ge0$,
  \begin{equation*}
    w_{0}^{-1}(K-\Delta_{M})^{s}w_{0}=
    (K-\Delta_{n}+V_{0}(r))^{s}.
  \end{equation*}
  Since for radial functions and $p<\infty$
  \begin{equation*}
    \textstyle
    \|\psi\|_{L^{p}(\mathbb{R}^{n})}^{p}=c_{n}
    \int_{0}^{\infty}|\psi(r)|^{p}r^{n-1}dr,
    \qquad
    \|\phi\|_{L^{p}(M)}^{p}=c_{n}
    \int_{0}^{\infty}|\phi(r)|^{p}h^{n-1}dr,
  \end{equation*}
  we get the identity
  \begin{equation*}
    \textstyle
    \|w^{\frac2p-1}(K-\Delta_{M})^{s}\phi\|_{L^{p}(M)}
    =
    \|(r^{k})^{\frac2p-1}(K-\Delta_{n}+V_{0})^{s}w_{0}^{-1}\phi\|
      _{L^{p}(\mathbb{R}^{n})},
    \qquad
    k:=\frac{m-n}{2}.
  \end{equation*}
  By repeating the proof of Lemma \ref{sob1} with the weight
  $(r^{k})^{\frac2p-1}$, we have
  \begin{equation*}
    \|(r^{k})^{\frac2p-1}(K-\Delta_{n}+V_{0})^{s}w_{0}^{-1}\phi\|
      _{L^{p}(\mathbb{R}^{n})}
    \simeq
    \|(r^{k})^{\frac2p-1}(K-\Delta_{n})^{s}w_{0}^{-1}\phi\|
      _{L^{p}(\mathbb{R}^{n})}
  \end{equation*}
  provided $K$ is large enough.
  % Since $(r^{k})^{\frac2p}=r^{\frac{m-n}{p}}$, this coincides,
  % up to a constant, with
  % \begin{equation*}
  %   \simeq
  %   \|r^{-k}(K-\Delta_{n})^{s}w_{0}^{-1}\phi\|
  %     _{L^{p}(\mathbb{R}^{m})}.
  % \end{equation*}
  Thus for $p=2$ we have proved that
  \begin{equation*}
    \textstyle
    \|(K-\Delta_{M})^{s}\phi\|_{L^{2}(M)}
    \simeq
    \|(K-\Delta_{n})^{s}w_{0}^{-1}\phi\|
      _{L^{2}(\mathbb{R}^{n})}
  \end{equation*}
  provided $K$ is large enough. However, the two norms
  are equivalent to the $H^{2s}$ norms on $M^{n}$ and on
  $\mathbb{R}^{n}$ respectively (on $M^{n}$ this follows
  easily from the spectral formula),
  thus we have proved that
  \begin{equation*}
    \|\phi\|_{H^{s}(M^{n})}\simeq
    \|w_{0}^{-1}\phi\|_{H^{s}(\mathbb{R}^{n})}
  \end{equation*}
  for all $s\ge0$ and all radial functions.
  In order to obtain \eqref{eq:equivLB2}
  it is sufficient to prove the equivalence
  \begin{equation*}
    \||x|^{k}v\|
      _{H^{s}(\mathbb{R}^{n})}
    \simeq
    \|v\|
      _{H^{s}(\mathbb{R}^{m})}
  \end{equation*}
  for all radial functions $v$, and this follows immediately
  from Lemma \ref{lem:dimchange}.

  Proceeding in a similar way, and
  using \eqref{eq1-lem2}, \eqref{eq1b-lem2}
  (with $c=V_{0}$), we obtain
  \begin{equation*}
    \||D_{M}|^{\pm \frac12}\bra{D_{M}}^{s}\phi\|_{L^{2}(M^{n})}
    \simeq
    \||D|^{\pm\frac12}\bra{D}^{s}(w_{0}^{-1}\phi)\|
        _{L^{2}(\mathbb{R}^{n})}
  \end{equation*}
  which is equivalent to
  \begin{equation*}
    \simeq
    \||D|^{\pm\frac12}(w_{0}^{-1}\phi)\|_{L^{2}(\mathbb{R}^{n})}
    +
    \||D|^{s\pm\frac12}(w_{0}^{-1}\phi)\|_{L^{2}(\mathbb{R}^{n})}.
  \end{equation*}
  In order to conclude the proof of \eqref{eq:equivLB2bis},
  it is now sufficient to apply to each term the equivalence
  for radial functions
  \begin{equation}\label{eq:chdimb}
    \||D|^{\sigma}v\|_{L^{2}(\mathbb{R}^{n})}
    \simeq
    \||D|^{\sigma}(r^{-k}v)\|_{L^{2}(\mathbb{R}^{m})}
  \end{equation}
  for $\sigma=s\pm \frac12$ and $\sigma=\pm \frac12$.
  When $\sigma\ge0$, this follows from
  Lemma \ref{lem:dimchange}. However, \eqref{eq:chdimb}
  is valid also when $0\ge\sigma\ge-1$; to check this fact
  we write \eqref{eq:chdimb} in the equivalent form
  \begin{equation*}
    \||x|^{-k}(1-\Delta_{n})^{\frac \sigma2}|x|^{k}w\|
        _{L^{2}(\mathbb{R}^{m})}
    \simeq
    \|(1-\Delta_{m})^{\frac \sigma2}w\|
        _{L^{2}(\mathbb{R}^{m})};
  \end{equation*}
  since
  $|x|^{-k}(1-\Delta_{n})|x|^{k}=
    (1-\Delta_{m}-\frac\ell{|x|^{2}})$
  where $\ell=k(k+n-2)$, we see that
  \eqref{eq:chdimb} is also equivalent to
  \begin{equation*}
    \textstyle
    \|(1-\Delta_{m}-\frac\ell{|x|^{2}})^{\frac \sigma2}w\|
        _{L^{2}(\mathbb{R}^{m})}
    \simeq
    \|(1-\Delta_{m})^{\frac \sigma2}w\|
        _{L^{2}(\mathbb{R}^{m})}
  \end{equation*}
  and this follows from Lemma \ref{lem:sobhom}.
\end{proof}

\section{Strichartz estimates for the perturbed equation}
\label{sec:strichartz}

We shall need a weighted version of Hardy's inequality:

\begin{proposition}\label{pro:hardy}
  Let $n\ge2$.
  Let $\alpha\in C^{1}(0,+\infty)$ be such that $\alpha>0$ and
  the integral $\beta(r):=\alpha(r)r^{n-1}
  \int_{0}^{r}\frac{ds}{\alpha(s)s^{n-1}}$ is finite
  for all $r>0$. Then the inequality
  \begin{equation}\label{eq:hardy}
    \int_{\mathbb{R}^{n}}
    \frac{|u|^{2}}{\beta(r)^{2}}
    \alpha(r) dx
    \le
    4\int_{\mathbb{R}^{n}}
    |\widehat{x}\cdot \nabla u|^{2}
    \alpha(r)
    dx,\qquad
    r=|x|
  \end{equation}
  is valid for all $u\in H^{1}_{loc}(\mathbb{R}^{n}\setminus0)$
  such that
  $\liminf_{r\to0^{+}}\frac{\alpha(r)}
  {\beta(r)}\int_{|x|=r}|u|^{2}dS=0$.
\end{proposition}

\begin{proof}%[of ...]
  By definition of $\beta$ we have for the radial
  derivative $\beta'=\widehat{x} \cdot \nabla \beta(|x|)$
  \begin{equation*}
    \beta'=\frac{n-1}{r}\beta+
    \frac{\alpha'}{\alpha}\beta
    +1
  \end{equation*}
  which implies the identity
  \begin{equation*}
    \nabla \cdot
    \left(
      \widehat{x}\frac{\alpha}{\beta}|u|^{2}
    \right)
    =
    -\frac{\alpha}{\beta^{2}}|u|^{2}
    +\frac{\alpha}{\beta}2\Re(u'\overline{u}).
  \end{equation*}
  Integrate over the difference of two balls
  $\Omega= B(0,R)\setminus B(0,r)$, $r<R$, to get
  \begin{equation*}
    \int_{\Omega}\frac{\alpha}{\beta^{2}}|u|^{2}dx
    =
    2\Re\int_{\Omega}\frac{\alpha}{\beta}u'\overline{u}dx
    +\int_{|x|=r}\frac{\alpha}{\beta}|u|^{2}dS
    -\int_{|x|=R}\frac{\alpha}{\beta}|u|^{2}dS.
  \end{equation*}
  Drop the last (negative) term and use Cauchy-Schwartz
  to obtain
  \begin{equation*}
    \int_{\Omega}\frac{\alpha}{\beta^{2}}|u|^{2}dx
    \le
    2\int_{|x|=r}\frac{\alpha}{\beta}|u|^{2}dS
    +
    4\int_{\Omega}\alpha|u'|^{2}dx
  \end{equation*}
  which implies \eqref{eq:hardy}, letting $r\to0^{+}$
  (on a suitable sequence) and $R\to+\infty$.
\end{proof}

\begin{corollary}\label{cor:hardy2}
  Let $\zeta\in C^{2}([0,\infty))$ with $\zeta\ge0$,
  $\zeta'>0$, $\zeta''\le0$, let $\epsilon>0$ and $n\ge2$.
  Then the inequality
  \begin{equation}\label{eq:hardy2}
    \int_{\mathbb{R}^{n}}
    [\zeta'+2 \epsilon \zeta]
    e^{-2 \epsilon r}r^{-(n-1)}\frac{|u|^{2}}{|x|^{2}}dx
    \le
    4\int_{\mathbb{R}^{n}}
    [\zeta'+2 \epsilon \zeta]
    e^{-2 \epsilon r}r^{-(n-1)}|\widehat{x}
    \cdot \nabla u|^{2}dx
  \end{equation}
  holds for any $u\in H^{1}_{loc}(\mathbb{R}^{n}\setminus0)$
  such that
  $\liminf_{r\to0^{+}}
    r^{-n}\int_{|x|=r}|u|^{2}dS=0$.
\end{corollary}

\begin{proof}%[of ...]
  Choose
  $\alpha=[\zeta'+2 \epsilon \zeta]e^{-2 \epsilon r}r^{1-n}$
  and apply Proposition \eqref{pro:hardy}. Notice that
  $(\alpha r^{n-1})'=
  [\zeta''-4 \epsilon^{2}\zeta]e^{-2 \epsilon r}\le0$
  so that $\alpha r^{n-1}$ is nonincreasing, and this
  implies $\beta\le r$.
  Thus we obtain inequality \eqref{eq:hardy2}, provided we can
  verify the condition $\alpha/\beta\int_{|x|=r}|u|^{2}\to 0$.
  By the assumptions on $\zeta$ we get
  $\zeta'(r)+2 \epsilon \zeta
  \le(1+2 \epsilon r)\zeta'(0)+\zeta(0)$; notice that
  $\zeta'(0)$ must be strictly positive. As a consequence,
  \begin{equation*}
    \frac{\alpha}{\beta}
    =
    \left(
    r^{n-1}\int_{0}^{r}
    \frac{e^{2 \epsilon s}ds}{\zeta'+2 \epsilon \zeta}
    \right)^{-1}
    \le
    C_{\epsilon}
    \left(
    r^{n-1}\int_{0}^{r}
    \frac{ds}{1+2 \epsilon s}
    \right)^{-1}
    \le
    C'_{\epsilon}r^{-n}
  \end{equation*}
  which concludes the proof.
\end{proof}

Consider now the equation on $\mathbb{R}^{n}$
\begin{equation}\label{eq:eqbase}
  u''+a(r)u'+\kappa^{2}u-c(r)u=f,\qquad
  r=|x|
\end{equation}
where $u(r)$ is radial and we write as usual
$u'=\widehat{x}\cdot \nabla u(|x|)$ for the radial derivative.
The function $a(r)$ will be smooth for $r>0$ but
singular at $r=0$, the model case being $a=(n-1)/r$.
Our next goal is to prove a suitable smoothing estimate
for solutions of \eqref{eq:eqbase}. In the following
we use the notation $L^{2}(\gamma(|x|) dx)$ to denote
the weighted $L^{2}$ space with norm
\begin{equation*}
  \|u\|_{L^{2}(\gamma(|x|) dx)}
  =\left(\int_{\mathbb{R}^{n}}|u(x)|^{2}\gamma(|x|)dx
  \right)^{\frac12}.
\end{equation*}

\begin{theorem}\label{the:smoo}
  Let $n\ge3$ and $\kappa\in \mathbb{C}$ with
  $\Im\kappa>0$.
  Let $a(r)\in C^{2}(0,+\infty)$,
  $c(r)\in C^{1}(0,+\infty)$ and denote with $A(r)$ a
  function such that $A'=a$ while $\gamma(r)= e^{A(r)}r^{1-n}$.

  Assume that $a(r)$ is bounded for large $r$,
  that $\lim_{r\to0^{+}}\gamma(r)>0$ exists,
  and that for some $0<\delta_{0}<1$ and some $C>0$
  the function
  \begin{equation*}
    Q(r)=\left(\frac{a'}{2}+\frac{a^{2}}{4}+c(r)\right)r
    +\frac{1-\delta_{0}}{4r}
  \end{equation*}
  satisfies the conditions
  \begin{equation*}
    0\le rQ(r)\le C,\qquad
    Q'(r)\le0.
  \end{equation*}

  Then any solution $u\in H^{2}_{loc}(\mathbb{R}^{n})$
  of equation \eqref{eq:eqbase} such that
  $u,u'\in L^{2}(\gamma(|x|)dx)$ satisfies the estimate
  \begin{equation}\label{eq:smooe}
    \||x|^{-1}u\|_{L^{2}(\gamma(|x|) dx)}
    \le
    4 \delta_{0}^{-1}
    \||x|f\|_{L^{2}(\gamma(|x|) dx)},
    \qquad
    \gamma(r)=r^{1-n}e^{A(r)}.
  \end{equation}
\end{theorem}

\begin{proof}%[of ...]
  Define new functions $v(r),g(r)$ via
  \begin{equation*}
    u(r)=e^{i \kappa r}e^{-A/2}v(r),\qquad
    f(r)=e^{i \kappa r}e^{-A/2}g(r)
  \end{equation*}
  and notice that $v(r)$ satisfies the equation
  \begin{equation*}
    v''+2 i \kappa v'-(\sigma(r)+c(r))v=g,
    \qquad
    \sigma(r):=\frac{a'}{2}+\frac{a^{2}}{4}.
  \end{equation*}
  Multiply the equation by $2\phi(r)\overline{v'}$,
  $\phi$ a weight to be chosen,
  and take the real part; using the identities
  \begin{equation*}
    \Re(2 \phi v'' \overline{v'})=
    (\phi|v'|^{2})'-\phi'|v'|^{2},
    \qquad
    \Re(4i\phi \kappa v' \overline{v'})=
    -4(\Im \kappa)\phi|v'|^{2}
  \end{equation*}
  and
  \begin{equation*}
    \Re (-2(\sigma+c )v \phi \overline{v'})=
    (-(\sigma+c)\phi|v|^{2})'+
    ((\sigma+c)\phi)'|v|^{2}
  \end{equation*}
  we obtain
  \begin{equation*}
    P'+
    [\phi'+4(\Im \kappa)\phi]|v'|^{2}-((\sigma+c)\phi)'|v|^{2}
    =\Re(-2 g \phi \overline{v'}),
  \end{equation*}
  where
  \begin{equation*}
    P=(\sigma+c)\phi|v|^{2}-\phi|v'|^{2}.
  \end{equation*}
  Notice that
  \begin{equation*}
    P'=\widehat{x}\cdot \nabla P=
    \nabla \cdot\left\{\widehat{x}P\right\}-
    \frac{n-1}{r}P,
  \end{equation*}
  thus we arrive at the identity
  \begin{equation}\label{eq:idbase}
    \textstyle
    \nabla \cdot\left\{\widehat{x}P\right\}
    +
    \left[\phi'+\frac{n-1}{r}\phi+ 4(\Im \kappa)\phi\right]
      \cdot |v'|^{2}
    -
    \left[((\sigma+c)\phi)'+\frac{n-1}{r}(\sigma+c)\phi\right]
      \cdot |v|^{2}
    =
    \Re(-2 g \phi \overline{v'}).
  \end{equation}
  We now choose
  \begin{equation*}
    \phi(r)=e^{-2(\Im \kappa)r}r^{-(n-2)}
  \end{equation*}
  which reduces the identity to
  \begin{equation}\label{eq:idbase2}
    \textstyle
    \nabla \cdot\left\{\widehat{x}P\right\}
    +
    [1+2(\Im \kappa)r]
      \cdot\frac{\phi}{r}|v'|^{2}
    +
    [2(\Im \kappa)(c+\sigma)r
    -((c+\sigma)r)']
      \cdot\frac{\phi}{r}|v|^{2}
    =
    \Re(-2 g \phi \overline{v'}).
  \end{equation}
  We integrate \eqref{eq:idbase2} on $B(0,R)\setminus B(0,r)$,
  $r<R$, and we check the behaviour of the boundary terms
  as $r\to0$, $R\to+\infty$. Near zero, we must prove that
  \begin{equation*}
    \textstyle
    \liminf_{r\to0^{+}}\int_{|x|=r}
    [(\sigma+c)\phi|v|^{2}-\phi|v'|^{2}] dS\le0
  \end{equation*}
  thus we can drop the second term $-\phi|v'|^{2}$ and
  focus on the first one.
  Recall that
  \begin{equation*}
    |v|^{2}=e^{A}e^{2(\Im \kappa)r}|u|^{2},\qquad
    \phi|v|^{2}=e^{A}r^{2-n}|u|^{2}
    \sim r|u|^{2}
    \ \text{near 0}\
  \end{equation*}
  since by assumption $e^{A}r^{1-n}\to C$ as $r\to0$.
  Noticing that the assumption on $Q$ implies
  $|\sigma+c|\le C r^{-2}$, we see that
  it is sufficient to prove
  \begin{equation}\label{eq:near0u}
    \textstyle
    \liminf_{r\to0^{+}}r^{-1}\int_{|x|=r}|u|^{2}dS=0.
  \end{equation}
  The assumption
  $u,u'\in L^{2}(\gamma dx)$ implies
  $u,u'\in L^{2}_{loc}(\mathbb{R}^{n})$ with the standard norm
  since $\gamma\sim C$ near 0,
  and hence, by the usual Hardy inequality, we have
  $u/r\in L^{2}_{loc}(\mathbb{R}^{n})$ which gives
  \eqref{eq:near0u}. For future reference
  we note also that
  \begin{equation}\label{eq:near0}
    \textstyle
    \liminf_{r\to0^{+}}\frac1r\int_{|x|=r}|u|^{2}dS=0
    \quad\implies\quad
    \liminf_{r\to0^{+}}
    r^{-n}
    \int_{|x|=r}|v|^{2}dS=0,
  \end{equation}
  by definition of $v$ and the assumption
  $e^{A}r^{1-n}\to C>0$. We then consider the boundary
  term on $\partial B(0,R)$ as $R\to+\infty$; we must prove that
  \begin{equation*}
    \textstyle
    \liminf_{R\to+\infty}\int_{|x|=R}
    [(\sigma+c)\phi|v|^{2}-\phi|v'|^{2}] dS\ge0
  \end{equation*}
  For the first term we write, recalling that
  $|\sigma+c|\le C r^{-2}$,
  \begin{equation*}
    |(\sigma+c)|\phi|v|^{2}=
    |(\sigma+c)|r \gamma|u|^{2}
    \le C r^{-1}\gamma|u|^{2}
  \end{equation*}
  and then the assumption
  $\gamma|u|^{2}\in L^{1}(\mathbb{R}^{n})$
  implies
  $\liminf_{R\to+\infty}R^{-1}\int_{|x|=R}\gamma|u|^{2}dS=0$.
  For the second term, we have
  \begin{equation*}
    \textstyle
    \int_{|x|=R}\phi|v'|^{2}dS
    \le
    CR\int_{|x|=R}(|\kappa|^{2}+a^{2})(|u|^{2}+|u'|^{2})\gamma(R)dS
  \end{equation*}
  and by the assumptions $u,u'\in L^{2}(\gamma dx)$ and
  $|a|\le C$ for large $r$ we have
  \begin{equation*}
    \liminf_{R\to+\infty}
    R\int_{|x|=R}(|\kappa|^{2}+a^{2})(|u|^{2}+|u'|^{2})\gamma(R)dS
    =0
  \end{equation*}
  as required.

  Thus we are in position to integrate \eqref{eq:idbase2}
  on $\mathbb{R}^{n}$:
  \begin{equation*}
    \textstyle
    \int
    [1+2(\Im \kappa)r]
      \cdot\frac{\phi}{r}|v'|^{2}dx
    +
    \int
    [2(\Im \kappa)(c+\sigma)r
    -((c+\sigma)r)']
      \cdot\frac{\phi}{r}|v|^{2}dx
    \le
    \int
    \Re(-2 g \phi \overline{v'})dx.
  \end{equation*}
  We estimate the right hand side by Cauchy-Schwartz
  and absorb a term at left, obtaining
  for any $0<\delta_{0}<1$
  \begin{equation}\label{eq:intbase}
    \textstyle
    \int
    [1-\delta_{0}+2(\Im \kappa)r]
      \cdot\frac{\phi}{r}|v'|^{2}dx
    +
    \int
    [2(\Im \kappa)(c+\sigma)r
    -((c+\sigma)r)']
      \cdot\frac{\phi}{r}|v|^{2}dx
    \le
    \frac1 {\delta_{0}}
    \int r \phi|g|^{2}dx
  \end{equation}
  Now we apply \eqref{eq:hardy2} of the previous
  Corollary with the choice $\zeta(r)=r$; note that the
  assumption on the behaviour of the function
  near 0 has already been checked in \eqref{eq:near0}.
  Recalling that $\phi/r=e^{-2(\Im \kappa)r}r^{1-n}$,
  this gives
  \begin{equation*}
    \textstyle
    \frac14
    \int[1+2(\Im \kappa)r]\frac{\phi}{r}\frac{|v|^{2}}{r^{2}}dx
    \le
    \int[1+2(\Im \kappa)r]\frac{\phi}{r}|v'|^{2}dx.
  \end{equation*}
  Using this inequality in \eqref{eq:intbase} we obtain,
  for any $0<\delta_{1}<1$,
  \begin{equation*}
    \textstyle
    \int
    [2(\Im \kappa)Q_{0}(r)
    -Q_{0}'(r)]
    \frac{\phi}{r}|v|^{2}dx
    \le
    \frac1 {\delta_{1}}
    \int r \phi|g|^{2}dx
  \end{equation*}
  where
  \begin{equation*}
    Q_{0}(r)=\frac{1-\delta_{1}}{4r}+(c+\sigma)r.
  \end{equation*}
  Since $\phi|v|^{2}=r \gamma|u|^{2}$ and
  $\phi|g|^{2}=r \gamma|f|^{2}$, this is equivalent to
  \begin{equation*}
    \textstyle
    \int
    [2(\Im \kappa)Q_{0}(r)
    -Q_{0}'(r)]
    |u|^{2}\gamma dx
    \le
    \frac1 {\delta_{1}}
    \int r^{2}|f|^{2}\gamma dx.
  \end{equation*}
  Choose $\delta_{1}=\delta_{0}/2$; we have
  by assumption
  \begin{equation*}
    Q_{0}(r)\ge Q(r)\ge0,\qquad
    -Q_{0}'(r)=
    -Q'(r)+\frac{\delta_{0}-\delta_{1}}{4r^{2}}
    \ge \frac{\delta_{0}}{8r^{2}}
  \end{equation*}
  and this concludes the proof.
\end{proof}

We now specialize the previous estimate to the resolvent equation
for a Laplace-Beltrami operator on the manifold $M$
with global metric $dr^{2}+h(r)^{2}d \omega^{2}_{\mathbb{S}^{n-1}}$,
restricted to radial functions:
\begin{equation}\label{eq:manifres}
  u''+(n-1)\frac{h'}{h}u'+\lambda^{2}u=f.
\end{equation}
We have the following result:

\begin{corollary}\label{cor:smoomanif}
  Let $n\ge3$,
  $\lambda\in \mathbb{C}$ with $\Im \lambda>0$,
  $h\in C^{2}([0,+\infty))$, with $h>0$ for $r>0$, $h(0)=0$, and
  $h'(0)=1$, and
  define the functions
  \begin{equation}\label{eq:defhtilde}
     \mu(r)=\left(\frac{h}{r}\right) ^{n-1},
     \qquad
     \widetilde{h}(r)=
     \frac{n-1}{2}\left(\frac{h''}{h}
       +\frac{n-3}{2}
       \frac{h'^{2}}{h^{2}}
     \right).
   \end{equation}
   Assume that
   $h_{\infty}:=\lim_{r\to+\infty}\widetilde{h}(r)\ge0$
   exists and that, for some $0<\delta_{0}<1$
   and $C>0$, the function
   \begin{equation}\label{eq:defP}
     P(r)=r(\widetilde{h}(r)-h_{\infty})+\frac{1-\delta_{0}}{4r}
   \end{equation}
   satisfies the conditions
   \begin{equation}\label{eq:condP}
     0\le rP(r)\le C,\qquad
     P'(r)\le0.
   \end{equation}
   Then any radial solution $u\in H^{2}_{loc}(\mathbb{R}^{n})$
   of equation \eqref{eq:manifres} such that
   $u,u'\in L^{2}(\mu(|x|)dx)$ satisfies the estimate
   \begin{equation}\label{eq:smooe2}
     \||x|^{-1}u\|_{L^{2}(\mu(|x|) dx)}
     \le
     4 \delta_{0}^{-1}
     \||x|f\|_{L^{2}(\mu(|x|) dx)}.
   \end{equation}
\end{corollary}

\begin{proof}%[of ...]
  Let $w(r)=\mu(r)^{\frac12}u(r)$ and
  $g(r)=\mu(r)^{\frac12}f(r)$,
  then $w(r)$ satisfies the equation
  \begin{equation*}
    \textstyle
    w''+\frac{n-1}{r}w'
    +\lambda^{2}w-
    \left(\widetilde{h}-\frac{(n-1)(n-3)}{4r^{2}}\right)w
    =g.
  \end{equation*}
  Setting $\lambda^{2}-h_{\infty}=\kappa^{2}$ with
  $\Im \kappa>0$ (possible since $h_{\infty}\ge0$),
  we rewrite the equation as
  \begin{equation*}
    \textstyle
    w''+\frac{n-1}{r}w'
    +\kappa^{2} w-
    \left(\widetilde{h}-h_{\infty}
    -\frac{(n-1)(n-3)}{4r^{2}}\right)w
    =g.
  \end{equation*}
  Now we can apply Theorem \ref{the:smoo} with the choices
  $a(r)=(n-1)/r$, $A=(n-1)\log r$, $\gamma(r)=1$,
  $c(r)=\widetilde{h}-h_{\infty}-(n-1)(n-3)/(4r^{2})$ so that
  $Q(r)\equiv P(r)$ as one checks immediately. Thus all
  the assumptions of the Theorem are satisfied and we
  get the estimate
  \begin{equation*}
    \||x|^{-1}w\|_{L^{2}(\mathbb{R}^{n})}\le
    4 \delta^{-1}
    \||x|g\|_{L^{2}(\mathbb{R}^{n})}
  \end{equation*}
  which coincides with \eqref{eq:smooe2}.
\end{proof}

If we apply the change of variables
\begin{equation*}
  u(r)=
  \frac{r^{k+\frac{n-1}{2}}}{h(r)^{\frac{n-1}{2}}}v(r),
  \qquad
  f(r)=
  \frac{r^{k+\frac{n-1}{2}}}{h(r)^{\frac{n-1}{2}}}g(r),
  \qquad
  k=0,1,2,\dots
\end{equation*}
we see that $u(r)$ solves \eqref{eq:manifres} if and only if
$v(r)$ solves the following equation,
which we shall regard as a radial equation on
$\mathbb{R}^{m}$:
\begin{equation}\label{eq:resV}
  v''+\frac{m-1}{r}v'+\lambda^{2}v-V(r)v=g,\qquad
  m=2k+n
\end{equation}
where
\begin{equation}\label{eq:defV}
  V(r)=
  \frac{n-1}{2}
  \left[
    \frac{h''}{h}
    +  \frac{n-3}{2}
    \left(
      \frac{h'^{2}}{h^{2}}-\frac{1}{r^{2}}
    \right)
  \right]
  +k(k+n-2)\left(
    \frac{1}{h^{2}}-\frac{1}{r^{2}}
  \right).
\end{equation}
If $h(r)$ satisfies the assumptions of
Corollary \ref{cor:smoomanif}, and
we apply the previous change of variables
in estimate \eqref{eq:smooe2}, we obtain:

\begin{corollary}\label{cor:smooV}
  Let $n\ge3$, $k\ge0$, $m=2k+n$,
  $\lambda\in \mathbb{C}$ with $\Im \lambda>0$, and
  let $h(r)$ be as in Corollary \ref{cor:smoomanif}.
  Then any radial solution $v\in H^{2}_{loc}(\mathbb{R}^{m})$
  of equation \eqref{eq:resV}-\eqref{eq:defV}
  such that
  $v,v'\in L^{2}(\mathbb{R}^{m})$ satisfies the estimate
  \begin{equation}\label{eq:smooeV}
    \||x|^{-1}v\|_{L^{2}(\mathbb{R}^{m})}\lesssim
    \||x|g\|_{L^{2}(\mathbb{R}^{m})}.
  \end{equation}
\end{corollary}

We notice that, defining
\begin{equation}\label{eq:potW}
  W(r)=V(r)-h_{\infty}
\end{equation}
equation \eqref{eq:resV} can be written
\begin{equation*}
  \Delta v+(\lambda^{2}-h_{\infty})v
  -W(r)v=g
\end{equation*}
where $\Delta$ is the Laplace operator on $\mathbb{R}^{m}$
(restricted to radial functions); moreover, it is easy to
check that the first part of assumption \eqref{eq:condP}
and the condition $h(r)\ge Cr$ for a $C>0$ imply
\begin{equation*}
  \frac{C'}{|x|^{2}}\ge W(r)\ge
  -\frac{(m-2)^{2}-\delta_{0}}{4|x|^{2}}.
\end{equation*}
By Hardy's inequality we obtain
\begin{equation*}
  \|\nabla u\|_{L^{2}(\mathbb{R}^{m})}\simeq
  (Hv,v).
\end{equation*}
Thus the operator
$H=-\Delta+W(r)$ is selfadjoint and positive
definite on $L^{2}(\mathbb{R}^{m})$, and by interpolation
and duality we have, as in Lemma \ref{lem:sobhom},
the equivalence of norms
\begin{equation}\label{eq:equivnorms}
  \|H^{s/2}v\|_{L^{2}(\mathbb{R}^{m})}
  \simeq
  \|v\|_{\dot H^{s}(\mathbb{R}^{m})},\qquad
  -1\le s\le1
\end{equation}
and analogously, for every $\nu>0$
(with a constant depending on $\nu$),
\begin{equation}\label{eq:equivnorms2}
  \|(\nu+H)^{s/2}v\|_{L^{2}(\mathbb{R}^{m})}
  \simeq
  \|v\|_{H^{s}(\mathbb{R}^{m})},\qquad
  -1\le s\le1.
\end{equation}

We can now apply Kato's theory to deduce, from the
resolvent estimate \eqref{eq:smooe2}, corresponding
smoothing estimates for the associated
evolution equations. In the terminology of
\cite{Kato65-b}, \cite{KatoYajima89-a}, estimate
\eqref{eq:smooe2} implies that multiplication
by $|x|^{-1}$ is \emph{supersmoothing} for the operator
$H$, and this implies the estimate
\begin{equation*}
  \||x|^{-1}e^{itH}f\|_{L^{2}(\mathbb{R}^{m+1})}
  \lesssim
  \|f\|_{L^{2}(\mathbb{R}^{m})}
\end{equation*}
for the Schr\"{o}dinger flow $e^{itH}$.
Using the appendix to Kato's theory developed in
\cite{DAncona14-a} we obtain the analogous result for
the wave flow:

\begin{theorem}\label{the:smooSWE}
  Let $n\ge3$, $k\ge0$, $m=2k+n$,
  let $h(r)$ be as in Corollary \ref{cor:smoomanif}
  and assume in addition $h(r)\ge cr$ for some $c>0$.
  Let $V(r)$ be the function \eqref{eq:defV}, $W=V-h_{\infty}$,
  and let $H$ be the selfadjoint
  nonnegative operator on $L^{2}(\mathbb{R}^{m})$
  given by $H=-\Delta+W(r)$.
  Then the wave flow $e^{it \sqrt{H}}$, restricted to
  radial functions, satisfies the smoothing estimate
  \begin{equation}\label{eq:smooWE}
    \||x|^{-1}e^{it \sqrt{H}}f\|
       _{L^{2}(\mathbb{R}^{m+1})}
    \lesssim \|f\|_{\dot H^{1/2}(\mathbb{R}^{m})},
  \end{equation}
  and, for any $\nu>0$,
  the Klein-Gordon flow $e^{it \sqrt{\nu+H}}$
  satisfies the smoothing estimate on radial functions
  \begin{equation}\label{eq:smooKG}
    \||x|^{-1}e^{it \sqrt{\nu+H}}f\|
       _{L^{2}(\mathbb{R}^{m+1})}
    \lesssim \|f\|_{H^{1/2}(\mathbb{R}^{m})}.
  \end{equation}
\end{theorem}

\begin{proof}%[of ...]
  By Theorem 2.4 in
  \cite{DAncona14-a}
  the operator $|x|^{-1}(H+\nu)^{-\frac14}$
  is supersmoothing with respect to $\sqrt{H+\nu}$.
  Then by Kato's theory we deduce the estimate
  \begin{equation*}
    \||x|^{-1}\sqrt{H+\nu}^{-\frac12}
        e^{-t \sqrt{H+\nu}}f\|_{L^{2}(\mathbb{R}^{m+1})}
    \lesssim
    \|f\|_{L^{2}(\mathbb{R}^{m})}.
  \end{equation*}
  When $\nu=0$, using \eqref{eq:equivnorms}, we obtain
  \eqref{eq:smooWE}, while for $\nu>0$, by
  \eqref{eq:equivnorms2}, we get \eqref{eq:smooKG}.
\end{proof}

Now, using the method of Rodnianski and Schlag
\cite{RodnianskiSchlag04-a}, it is a simple task to deduce the
full range of non-endpoint Strichartz estimates for
the wave and Klein-Gordon equations associated to the
operator $H$. The following is the main result of
this section; we sum up in the statement the previous
assumptions and notations.

\begin{theorem}\label{the:strichartz}
  Let
  $h\in C^{2}([0,+\infty))$ with $h(r)\ge cr$ for some $c>0$,
  $h(0)=0$ and $h'(0)=1$. Define for $n\ge3$
  \begin{equation}\label{eq:htilde2b}
     \widetilde{h}(r)=
     \frac{n-1}{2}\left(\frac{h''}{h}
       +\frac{n-3}{2}
       \frac{h'^{2}}{h^{2}}
     \right).
  \end{equation}
  Assume that
  $h_{\infty}:=\lim_{r\to+\infty}\widetilde{h}(r)\ge0$
  exists and that, for some $0<\delta_{0}<1$
  and $C>0$, the function
  \begin{equation}\label{eq:defP2}
    P(r)=r(\widetilde{h}(r)-h_{\infty})+\frac{1-\delta_{0}}{4r}
  \end{equation}
  satisfies the conditions
  \begin{equation}\label{eq:condP2}
    0\le rP(r)\le C,\qquad
    P'(r)\le0.
  \end{equation}
  Finally, let $k\ge0$, define $V(r)$ as
  \begin{equation}\label{eq:defV2b}
   V(r)=
   \frac{n-1}{2}
   \left[
     \frac{h''}{h}
     +  \frac{n-3}{2}
     \left(
       \frac{h'^{2}}{h^{2}}-\frac{1}{r^{2}}
     \right)
   \right]
   +k(k+n-2)\left(
     \frac{1}{h^{2}}-\frac{1}{r^{2}}
   \right),
  \end{equation}
  and let $H$ be the selfadjoint operator on
  $L^{2}(\mathbb{R}^{m})$, with $m=2k+n$,
  defined by $H=-\Delta+W(|x|)$, $W(r):=V(r)-h_{\infty}$.

  Then the wave flow $e^{it \sqrt{H}}$
  on $\mathbb{R}^{t}\times \mathbb{R}^{m}$
  satisfies the following Strichartz estimates:
  for radial $f$,
  \begin{equation}\label{eq:striWE}
    \||D|^{\frac1q-\frac 1p}
        e^{it \sqrt{H}}f\|_{L^{p}_{t}L^{q}_{x}}
    \lesssim
    \|f\|_{\dot H^{\frac12}(\mathbb{R}^{m})},
    \qquad
    |D|=(-\Delta)^{\frac12},
  \end{equation}
  provided $(p,q)$ satisfy
  \begin{equation}\label{eq:admpq}
    \frac2p+\frac {m-1}{q}=\frac{m-1}{2},\qquad
    2<p\le \infty,\qquad
    2\le q < \frac{2(m-1)}{m-3}
  \end{equation}
  while for fixed $\nu>0$
  the Klein-Gordon flow $e^{it \sqrt{H+\nu}}$
  on $\mathbb{R}^{t}\times \mathbb{R}^{m}$
  satisfies, for radial $f$,
  \begin{equation}\label{eq:striKG}
    \|\bra{D}^{\frac1q-\frac 1p}
        e^{it \sqrt{\nu+H}}f\|_{L^{p}_{t}L^{q}_{x}}
    \lesssim
    \|f\|_{ H^{\frac12}(\mathbb{R}^{m})},
    \qquad
    \bra{D}=(1-\Delta)^{\frac12},
  \end{equation}
  provided $(p,q)$ satisfy either \eqref{eq:extadmWE} or
  \begin{equation}\label{eq:admpqS}
    \frac2p+\frac {m}{q}=\frac{m}{2},\qquad
    2<p\le \infty,\qquad
    2 \le q <\frac{2m}{m-2}.
  \end{equation}
\end{theorem}

\begin{proof}
  By Duhamel's formula one can
  represent the flow $u(t,x)=e^{it \sqrt{H}}f$
  in terms of the unperturbed flow as
  \begin{equation*}
    e^{it \sqrt{H}}f=
    \cos(t |D|)f+
    i\sin(t |D|)|D|^{-1}\sqrt{H} f-
    \int_{0}^{t}\frac{\sin((t-s) |D|)}
    {|D|}W(r)uds.
  \end{equation*}
  For the first two terms, by the standard Strichartz
  estimates for the unperturbed wave equation we have
  \begin{equation*}
    \||D|^{\frac1q-\frac 1p}
        e^{it |D|}f\|_{L^{p}_{t}L^{q}_{x}}
    \lesssim
    \|f\|_{\dot H^{\frac12}}
  \end{equation*}
  and
  \begin{equation*}
    \||D|^{\frac1q-\frac 1p}
        |D|^{-1}e^{it |D|}\sqrt{H}f\|_{L^{p}_{t}L^{q}_{x}}
    \lesssim
    \|\sqrt{H}f\|_{\dot H^{-\frac12}}
    \lesssim
    \|f\|_{\dot H^{\frac12}}
  \end{equation*}
  with $(p,q)$ as in \eqref{eq:admpq}.
  In order to handle the Duhamel term, one uses
  the following mixed estimate for the free flow
  \begin{equation}\label{eq:duha}
    \left\||D|^{\frac1q-\frac 1p}
      \int_{0}^{t}\frac{e^{i(t-s)|D|}}{|D|}
          F(s,x)ds\right\|_{L^{p}_{t}L^{q}_{x}}
    \lesssim
    \||x| F\|_{L^{2}(\mathbb{R}^{m+1})}.
  \end{equation}
  This estimate is proved in a standard way as follows:
  first by the Christ-Kiselev lemma the estimate is
  equivalent to the similar estimate for the
  untruncated integral (provided $p>2$, which excludes
  the endpoint case); then the estimate is split into
  the homogeneous estimate for the free flow
  \begin{equation*}
    \||D|^{\frac1q-\frac 1p}e^{it|D|}|D|^{-\frac12}f\|
        _{L^{p}_{t}L^{q}_{x}}\lesssim
    \|f\|_{L^{2}}
  \end{equation*}
  composed with the dual smoothing estimate for the free flow
  \begin{equation*}
    \textstyle
    \|\int |D|^{-\frac12}e^{-is|D|}G(s,x)ds\|_{L^{2}}
    \lesssim
    \||x|G(t,x)\|_{L^{2}_{t}L^{2}_{x}}
  \end{equation*}
  (dual of \eqref{eq:smooWE} for the unperturbed
  wave equation). Now,
  plugging $F=Wu$ inside the right hand side of \eqref{eq:duha}
  and noticing that $|W|\le C|x|^{-2}$, we obtain
  also for the Duhamel term
  \begin{equation*}
    \left\||D|^{\frac1q-\frac 1p}
      \int_{0}^{t}\frac{e^{i(t-s)|D|}}{|D|}
          Wuds\right\|_{L^{p}_{t}L^{q}_{x}}
    \lesssim
    \||x| Wu\|_{L^{2}(\mathbb{R}^{m+1})}
    \le
    C\||x|^{-1}u\|_{L^{2}(\mathbb{R}^{m+1})}
  \end{equation*}
  which is bounded by $\|f\|_{L^{2}(\mathbb{R}^{m})}$
  using the smoothing estimate \eqref{eq:smooWE}.
  The three estimates together give \eqref{eq:striWE}.
  The proof for Klein-Gordon is identical; the
  estimates for the free flow which are required in the proof
  have the form
  \begin{equation*}
    \|\bra{D}^{\frac1q-\frac 1p}
        e^{it \sqrt{1-\Delta}}f\|_{L^{p}_{t}L^{q}_{x}}
    \lesssim
    \|f\|_{ H^{\frac12}(\mathbb{R}^{m})}.
  \end{equation*}
  Such estimates hold both if the couple $(p,q)$ is
  wave admissible, i.e. satisfies \eqref{eq:admpq}, and
  if it is Schr\"{o}dinger admissible, i.e. satisfies
  \eqref{eq:admpqS}.
  A complete proof in the second case can be found for
  instance in the Appendix of \cite{DAnconaFanelli08-a}.
  On the other hand for the first case the proof follows
  from the estimate
  \begin{equation*}
    j\ge1,\quad
    \phi_{j}\in \mathscr{S},
    \quad
    \spt \widehat{\phi_{j}}=\{|\xi|\sim 2^{j}\}
    \quad\implies\quad
    \|e^{it \sqrt{1-\Delta}}\phi_{j}\|
        _{L^{\infty}(\mathbb{R}^{m})}
    \lesssim
    |t|^{-\frac{m-1}{2}}2^{\frac{m+1}{2}}
  \end{equation*}
  (due to Brenner \cite{Brenner85-a})
  by the standard Ginibre-Velo procedure; note that we do not
  need the endpoint estimate.
\end{proof}

Using again
the Christ-Kiselev lemma and a $TT^{*}$ argument, we
deduce in a standard way the
nonhomogeneous Strichartz estimates from the previous
homogeneous estimates, at least in the non endpoint case.
We obtain

\begin{corollary}\label{cor:nonhstr}
  With the notations and the assumptions of the previous
  two Theorems, one has the estimate
  \begin{equation}\label{eq:strnhWE}
   \left\||D|^{\frac1q-\frac 1p}
     \int_{0}^{t}\frac{e^{i(t-s) \sqrt{H}}}{\sqrt{H}}F(s)ds
   \right\|_{L_t^pL_x^q}\lesssim
   \||D|^{-\frac1{\widetilde{q}}+\frac1{\widetilde{p}}}
        F\|_{L_t^{\widetilde{p}'}
        L_x^{\widetilde{q}'}}
  \end{equation}
  for all $F(t,x)$ radial in the space variable, and all
  couples $(p,q)$ and $(\widetilde{p},\widetilde{q})$
  as in \eqref{eq:admpq}. Similarly, we have
  for $\nu>0$
  \begin{equation}\label{eq:strnhKG}
   \left\|\bra{D}^{\frac1q-\frac 1p}
     \int_{0}^{t}\frac{e^{i(t-s) \sqrt{\nu+H}}}
       {\sqrt{\nu+H}}F(s)ds
   \right\|_{L_t^pL_x^q}\lesssim
   \|\bra{D}^{-\frac1{\widetilde{q}}+\frac1{\widetilde{p}}}
       F\|_{L_t^{\widetilde{p}'}
        L_x^{\widetilde{q}'}}
  \end{equation}
  for all $F(t,x)$ radial in the space variable, and all
  couples $(p,q)$ and $(\widetilde{p},\widetilde{q})$
  satisfying either \eqref{eq:extadmWE} or \eqref{eq:admpqS}.
\end{corollary}

\begin{remark}[]\label{rem:KGadmiss}
  In the Klein-Gordon case, by interpolation one can obtain
  a wider range of admissible couples $(p,q)$.
  We omit the details since in
  the following we shall only need the wave admissible
  case.
\end{remark}

\begin{remark}[]\label{rem:sobstricest}
  Recall that we can write
  the fractional Sobolev embedding on $\mathbb{R}^{m}$
  in the form
  \begin{equation*}
    \||D|^{\frac mr}v\|_{L^{r}}\lesssim
    \||D|^{\frac mq}v\|_{L^{q}},
    \qquad
    1<q\le r<\infty
  \end{equation*}
  which includes essentially all cases with the exception
  of some endpoints. On the other hand we can write
  \eqref{eq:striWE} in the equivalent form
  \begin{equation*}
    \||D|^{\frac1p+\frac mq-\frac {m-1}2}
        e^{it \sqrt{H}}f\|_{L^{p}_{t}L^{q}_{x}}
    \lesssim
    \|f\|_{\dot H^{\frac12}(\mathbb{R}^{m})}.
  \end{equation*}
  Nesting the two we obtain immediately the following
  \emph{Sobolev-Strichartz estimates}:
  \begin{equation}\label{eq:sobstriWE}
    \textstyle
    \||D|^{\frac1p+\frac mr-\frac{m-1}{2}}
        e^{it \sqrt{H}}f\|_{L^{p}_{t}L^{r}_{x}}
    \lesssim
    \|f\|_{\dot H^{\frac12}(\mathbb{R}^{m})}
  \end{equation}
  provided the couple $(p,r)$ satisfies
  \begin{equation}\label{eq:extadmWE}
    0<\frac1r\le \frac12-\frac{2}{m-1}\frac1p,
    \qquad
    2<p\le\infty.
  \end{equation}
  A similar extension of the range holds then also for the
  nonhomogeneous estimate \eqref{eq:strnhWE}, where
  we can replace $(p,q)$ and $(\widetilde{p},\widetilde{q})$
  with couples $(p,r)$ and $(\widetilde{p},\widetilde{r})$
  satisfying the extended condition \eqref{eq:extadmWE}.

  Similar extensions hold also for the Klein-Gordon
  equation. In the following, we shall only use
  the case of wave-type estimates. Recall that
  \begin{equation*}
    \|v\|_{L^{p}}\lesssim \|\bra{D}^{\frac mq}v\|_{L^{q}}
    \quad\text{for all}\quad q\le p<\infty
  \end{equation*}
  and together with the fractional Sobolev embedding
  this implies the inequality
  \begin{equation*}
    \|\bra{D}^{\frac mr}v\|_{L^{r}}\lesssim
    \|\bra{D}^{\frac mq}v\|_{L^{q}},
    \qquad
    1<q\le r<\infty.
  \end{equation*}
  Thus we can proceed exactly as for \eqref{eq:sobstriWE}
  and we obtain the extended estimates
  \begin{equation}\label{eq:sobstriKG}
    \|\bra{D}^{\frac1p+\frac mr-\frac{m-1}{2}}
        e^{it \sqrt{\nu+H}}f\|_{L^{p}_{t}L^{r}_{x}}
    \lesssim
    \|f\|_{ H^{\frac12}(\mathbb{R}^{m})},
  \end{equation}
  provided the couple $(p,r)$ satisfies
  \eqref{eq:extadmWE}. Note that we can obtain
  Sobolev-Strichartz estimates also in the Schr\"{o}dinger
  admissible range of indices, but we shall not need this.

  In a similar way, in estimate \eqref{eq:strnhKG} we can replace
  $(p,q)$ and $(\widetilde{p},\widetilde{q})$
  with any couples $(p,r)$ and $(\widetilde{p},\widetilde{r})$
  satisfying the extended condition \eqref{eq:extadmWE}.

 \end{remark}

\section{The fixed point argument}\label{sec:fixedpoint}

We begin by recalling some basic nonlinear estimates
for later use. The first one is a well-known H\"{o}lder
inequality for fractional derivatives:

\begin{lemma}[Kato-Ponce]\label{lem:katoponce}
  For any test functions $u,v$, any $s\ge0$ and $1<p<\infty$ one has
  \begin{equation}\label{eq:poncekato}
    \||D|^{s}(uv)\|_{L^{p}}\lesssim
    \||D|^{s}u\|_{L^{p_{1}}}\|v\|_{L^{p_{2}}}+
    \|u\|_{L^{p_{3}}}\||D|^{s}v\|_{L^{p_{4}}}
  \end{equation}
  and
  \begin{equation}\label{eq:poncekato2}
    \|\bra{D}^{s}(uv)\|_{L^{p}}\lesssim
    \|\bra{D}^{s}u\|_{L^{p_{1}}}\|v\|_{L^{p_{2}}}+
    \|u\|_{L^{p_{3}}}\|\bra{D}^{s}v\|_{L^{p_{4}}}
  \end{equation}
  provided $p_{1},p_{2},p_{3},p_{4}\in]1,\infty]$ satisfy
  $\frac1p=\frac{1}{p_{1}}+\frac{1}{p_{2}}=
         \frac{1}{p_{3}}+\frac{1}{p_{4}}$
\end{lemma}

In particular, this gives
\begin{equation*}
  \|u^{3}\|_{\dot H^{s}_{p}}\lesssim
  \|u\|_{\dot H^{s}_{r}}\|u\|_{L^{q}}^{2}
\end{equation*}
provided $p^{-1}=r^{-1}+2q^{-1}$.

The second Lemma is a standard fractional Moser type inequality:

\begin{lemma}\label{lem:coif}
  Assume $F(r)$ is in $C^{N}(\mathbb{R})$, $N\ge1$ integer,
  and let $0<s<N$, $1<p<\infty$. If
  $F(0)=0$, then there exists
  a function $\phi(r)$ such that
  for any test functions
  $u,v$ one has
  \begin{equation}\label{eq:compos}
    \|F(u)\|_{ H^{s}_{p}}\le
    \phi(\|u\|_{L^{\infty}})\|u\|_{ H^{s}_{p}},
  \end{equation}
  \begin{equation}\label{eq:compos2}
    \|F(u)-F(v)\|_{ H^{s}_{p}}\le
    \phi(R)\left[\|u-v\|_{ H^{s}_{p}}+\|u-v\|_{ L^{\infty}}\right],
  \end{equation}
  where $R=\|u\|_{L^{\infty}}+\|v\|_{L^{\infty}}
      +\|u\|_{H^{s}_{p}}+\|v\|_{H^{s}_{p}}$.
\end{lemma}

The third Lemma is a Strauss type inequality, i.e., an
improved weighted Sobolev embedding for radial
functions. For a proof and a comprehensive treatment
of such inequalities, we refer to \cite{DAnconaLuca12-a}.

\begin{lemma}\label{lem:radial}
  Let $m\ge2$,
  $1< p\le q< \infty$ and $\frac 1p-\frac 1q\le s<\frac mp$.
  For any radial function
  $u(x)$ on $\mathbb{R}^{m}$ one has
  \begin{equation}\label{eq:strauss}
    \||x|^{\frac mp-\frac mq-s}u\|_{L^{q}}\lesssim
    \|u\|_{\dot H^{s}_{p}}.
  \end{equation}
  If the condition on $s$ is restricted to
  $\frac 1p-\frac 1q< s<\frac mp$, then
  the previous estimate holds for all $1\le p\le q\le \infty$.
 \end{lemma}

 The following consequence of Lemma \ref{lem:radial}
 will be a crucial ingredient in the proof of the main result:

% \begin{lemma}\label{lem:radial3}
%   For any $1\le p\le q\le \infty$, any $\frac1p-\frac1q\le s<m/p-[\sigma]-1$, $\sigma\ge0$
%   and any radial function
%   $u(x)$ on $\mathbb{R}^{m}$, $m\ge2$, one has
%   \begin{equation}\label{eq:strauss3}
%     \||x|^{\frac mp-\frac mq-s}u\|_{\dot H^{\sigma}_{q}}\lesssim
%     \|u\|_{\dot H^{s+\sigma}_{p}}.
%   \end{equation}
%   If $\sigma$ is integer the range can be extended to
%   $\frac 1p-\frac 1q \le  s<m/p-[\sigma]$.
% \end{lemma}

\begin{lemma}[]\label{lem:radial33}
  Let $m\ge2$, $1< p\le q< \infty$, $\sigma\ge0$ and
  $\frac1p-\frac1q\le s<\frac mp-\sigma$. Assume the function
  $\gamma(r)\in C^{[\sigma]+1}(]0,+\infty[,\mathbb{R})$
  satisfies for $r>0$
  \begin{equation}\label{eq:decaygamma}
    |\gamma^{(j)}(r)|\lesssim r^{\frac mp-\frac mq-s-j},
    \qquad
    j=0,\dots,[\sigma]+1.
  \end{equation}
  Then for any radial function $u(x)$ on $\mathbb{R}^{m}$
  one has the estimate
  \begin{equation}\label{eq:radialgamma}
    \|\gamma(|x|)u\|_{\dot H^{\sigma}_{q}}\lesssim
    \|u\|_{\dot H^{s+\sigma}_{p}}.
  \end{equation}
\end{lemma}

\begin{proof}%[of ...]
  Write $\gamma(r)=\rho(r)|x|^{\frac mp-\frac mq-s}$
  so that
  \begin{equation*}
    \rho(r):=\gamma(r)|x|^{-\frac mp+\frac mq+s}
    \quad\implies\quad
    |\rho^{(j)}(r)|\lesssim r^{-j},
    \qquad
    j=0,\dots,[\sigma]+1.
  \end{equation*}
  We shall prove the estimate
  \begin{equation}\label{eq:theclaim}
    \||D|^{\sigma}(\rho(r)|x|^{\frac mp-\frac mq-z}u)\|_{L^{q}}
    \lesssim
    \||D|^{\sigma+z}u\|_{L^{p}}
  \end{equation}
  for all $1<p\le q<\infty$, $\sigma\ge0$ and all $z\in \mathbb{C}$
  in the complex strip
  $\frac1p-\frac1q\le \Re z<\frac mp-\sigma$. Note that
  the right hand side is equivalent to the
  $\dot H^{\Re z+\sigma}_{q}$ norm, by the well known property
  (see e.g.~ \cite{SikoraWright01-a} or
  \cite{CacciafestaDAncona12-a})
  \begin{equation*}
    \||D|^{iy}v\|_{L^{q}}\simeq \|v\|_{L^{q}},
    \qquad
    1<q<\infty.
  \end{equation*}
  When $\sigma$ is a nonnegative integer,
  the claim is proved directly writing
  \begin{equation*}
    \textstyle
    \||D|^{\sigma}(\rho(r)|x|^{\frac mp-\frac mq-z}u)\|_{L^{q}}
    \simeq
    \sum_{|\alpha|=\sigma}
    \|\partial^{\alpha}
      (\rho(r)|x|^{\frac mp-\frac mq-z}u)\|_{L^{q}},
  \end{equation*}
  expanding the derivatives by the chain rule,
  and applying to each term estimate \eqref{eq:strauss}.

  Consider now the case of a real $\sigma>0$. By complex
  interpolation between the integer cases
  \begin{equation*}
    \textstyle
    \||D|^{[\sigma]}(\rho(r)|x|^{\frac mp-\frac mq-z}u)\|_{L^{q}}
    \lesssim
    \|D^{s+[\sigma]}u\|_{L^{p}},
    \qquad
    \frac1p-\frac1q\le \Re z<\frac mp-[\sigma]
  \end{equation*}
  and
  \begin{equation*}
    \textstyle
    \||D|^{[\sigma]+1}(\rho(r)|x|^{\frac mp-\frac mq-z}u)\|_{L^{q}}
    \lesssim
    \|D^{s+[\sigma]+1}u\|_{L^{p}},
    \qquad
    \frac1p-\frac1q\le \Re z<\frac mp-[\sigma]-1
  \end{equation*}
  we obtain that \eqref{eq:theclaim} is true
  provided
  \begin{equation}\label{eq:firstrange}
    \textstyle
    \frac1p-\frac1q\le \Re z<\frac mp-[\sigma]-1.
  \end{equation}
  On the other hand, if we first use Sobolev embedding
  \begin{equation*}
    \textstyle
    \||D|^{\sigma}(\rho(r)|x|^{\frac mp-\frac mq-z}u)\|_{L^{q}}
    \lesssim
    \||D|^{[\sigma]+1}(\rho(r)|x|^{\frac mp-\frac mq-z}u)\|_{L^{r}},
    \qquad
    [\sigma]+1-\frac mr=\sigma-\frac mq
  \end{equation*}
  and then apply again the inequality for the integer case,
  we obtain that \eqref{eq:theclaim} is true for $\Re z$ in the
  range
  \begin{equation*}
    \textstyle
    \frac1p-\frac1r\le
    \Re z+\sigma-([\sigma]+1)<\frac mp-([\sigma]+1)
  \end{equation*}
  which is equivalent to
  \begin{equation}\label{eq:secondrange}
    \textstyle
    \frac1p-\frac1q+(1-\{\sigma\})\frac{m-1}{m}\le
    \Re z < \frac mp-\sigma,
    \qquad
    \{\sigma\}:=\sigma-[\sigma].
  \end{equation}
  Now we define the analytic family of operators
  \begin{equation*}
    T_{z}v:=|D|^{\sigma}(\rho(r)|x|^{\frac mp-\frac mq-z}
        |D|^{-z-\sigma}v)
  \end{equation*}
  and we note that we have proved that $T_{z}:L^{p}\to L^{q}$
  is bounded for $\Re z$ in the range \eqref{eq:firstrange} and
  also in the range \eqref{eq:secondrange}.
  By Stein-Weiss interpolation we obtain that $T_{z}$ is
  bounded for all $\Re z$ in the range
  $\frac1p-\frac1q\le\Re z<\frac mp-\sigma$ as claimed, and this
  concludes the proof.
\end{proof}

We are now ready to prove the main result of this section.

Let $n\ge3$ and let $h\in C^{[\frac{n-1}{2}]+2}([0,+\infty))$
with the properties
\begin{equation}\label{eq:prophr}
  h(0)=h''(0)=0,\qquad
  h'(0)=1,\qquad
  h(r)>cr
  \ \text{for some $c>0$}.
\end{equation}
 Define
\begin{equation}\label{eq:htilde2}
   \widetilde{h}(r)=
   \frac{n-1}{2}\left(\frac{h''}{h}
     +\frac{n-3}{2}
     \frac{h'^{2}}{h^{2}}
   \right).
\end{equation}
Assume that
$h_{\infty}:=\lim_{r\to+\infty}\widetilde{h}(r)\ge0$
exists and that, for some $0<\delta_{0}<1$
and $C>0$, the function
\begin{equation}\label{eq:defP2b}
  P(r)=r(\widetilde{h}(r)-h_{\infty})+\frac{1-\delta_{0}}{4r}
\end{equation}
satisfies the conditions
\begin{equation}\label{eq:condP2base}
  0\le rP(r)\le C,\qquad
  P'(r)\le0.
\end{equation}
Finally, let $k\ge1$ and define the potential $V(r)$ as
\begin{equation}\label{eq:defVc}
  V(r)=
  \frac{n-1}{2}
  \left[
    \frac{h''}{h}
    +  \frac{n-3}{2}
    \left(
      \frac{h'^{2}}{h^{2}}-\frac{1}{r^{2}}
    \right)
  \right]
  +k(k+n-2)\left(
    \frac{1}{h^{2}}-\frac{1}{r^{2}}
  \right).
\end{equation}
and assume that
\begin{equation}\label{eq:assderV}
  |V^{(j)}(r)|\lesssim r^{-1}
  \quad\text{for large $r$,}
  \qquad
  \textstyle
  1\le j\le[\frac{n-1}{2}].
\end{equation}
We consider the following Cauchy problem on
$\mathbb{R}_{t}\times \mathbb{R}_{x}^{m}$:
\begin{equation}\label{eq:cauchykey}
  \textstyle
  \psi_{tt}-\Delta \psi+V(|x|)\psi
    =\alpha(|x|)^2Z\bigl(\beta(|x|)\psi\bigr)\cdot\psi^3,
\end{equation}
\begin{equation}\label{eq:cauchykeydata}
  \psi(0,x)=f(x), \qquad \partial_t\psi(0,x)=g(x).
\end{equation}
with spherically symmetric data.

\begin{theorem}[Global well posedness with small data]
  \label{the:keyWE}
  Let $n\ge3$, $k\ge1$ and $m=2k+n$.
  Let $h\in C^{[\frac{n-1}{2}]+3}([0,+\infty))$
  be a real valued function satisfying
  \eqref{eq:prophr}, \eqref{eq:condP2base}
  while the potential $V(r)$ defined by \eqref{eq:defVc}
  satisfies \eqref{eq:assderV}.
  Moreover, let
  $\alpha,\beta\in C^{[\frac{n-1}{2}]+1}([0,+\infty))$
  be such that
  \begin{equation}\label{eq:assalbe}
    \textstyle
    r|\alpha^{(j)}(r)|+|\beta^{(j)}(r)|\lesssim r^{k-j},
    \qquad
    0\le j\le[\frac{n-1}{2}]+1.
  \end{equation}
  Consider the Cauchy problem
  \eqref{eq:cauchykey}, \eqref{eq:cauchykeydata} on
  $\mathbb{R}_{t}\times \mathbb{R}_{x}^{m}$
  with spherically symmetric data
  and $Z\in C^{[\frac{n-1}{2}]+1}(\mathbb{R},\mathbb{R})$.

  In the case $h_{\infty}>0$, if
  $\|f\|_{H^{\frac{n}{2}}}
     +\|g\|_{H^{\frac{n}{2}-1}}$
  is sufficiently small, Problem
  \eqref{eq:cauchykey}, \eqref{eq:cauchykeydata} has a
  unique global solution
  $\psi\in L^{\infty}H^{\frac{n}{2}}
    \cap CH^{\frac{n}{2}}
    \cap L^{p}H^{\frac{n-1}{2}}_{q} $,
  where
  $p=\frac{4(m+1)}{m+3}$, $q=\frac{4m(m+1)}{2m^{2}-m-5}$.

  In the case $h_{\infty}=0$, if
  $\||D|^{\frac12}f\|_{H^{\frac{n-1}{2}}}
     +\||D|^{-\frac12}g\|_{H^{\frac{n-1}{2}}}$
  is sufficiently small, Problem
  \eqref{eq:cauchykey}, \eqref{eq:cauchykeydata} has a unique
  global solution $\psi$ with
  $|D|^{\frac12}\psi\in L^{\infty}H^{\frac{n-1}{2}}
    \cap CH^{\frac{n-1}{2}}$ and
  $\psi\in\cap L^{p}H^{\frac{n-1}{2}}_{q}$,
  with $p,q$ as before.
\end{theorem}

\begin{proof}
  It is clear that the assumptions of
  Theorem \ref{the:strichartz} are satisfied.
  In particular, the function $h_{1}=\frac{h(r)-r}{r^{3}}$
  is of class $C^{[\frac{n-1}{2}]}$ as it can be verified
  by direct computation using properties \eqref{eq:prophr},
  and writing $h(r)=r+r^{3}h_{1}(r)$
  one checks easily that the potential
  $V(r)$ is of class $C^{[\frac{n-1}{2}]}$ also at the origin.
  In view of
  \eqref{eq:assderV}, we see that the assumptions
  of Lemmas \ref{lem:sobhom}, \ref{sob1},
  \ref{lem:irrelevantM} and \ref{sob2} are satisfied with
  the choice $c(x)=V(|x|)$.
  Thus we are in position to use the Strichartz estimates
  from that Theorem, and also the consequences in
  Corollary \ref{cor:nonhstr} and
  Remark \ref{rem:sobstricest}.

  Note also that the proof of the two cases $h_{\infty}=0$
  and $>0$ is almost identical; indeed, the Strichartz
  estimates that we use in the following
  are valid both in the wave and in the Klein-Gordon
  case. We shall perform
  the proof only in the first (slightly harder) case
  and leave to the reader to check that the argument works
  also in the second case with minimal modifications.

  Using the notations $|D_{V}|=(-\Delta+V)^{\frac12}$
  and $\bra{D_{V}}=(1-\Delta+V)^{\frac12}$,
  we define the nonlinear map
  \begin{equation*}
    \textstyle
    \Lambda(\psi(t,r)):=
    \cos (t|D_V|)f+
    \sin(t|D_V|)|D_V|^{-1}g+
    \ \Box_V^{-1}F
  \end{equation*}
  where
  \begin{equation*}
    \textstyle
    \Box_V^{-1}F=\int_{0}^{t}\frac{\sin((t-s)|D_V|)}{|D_V|}F(s)ds,
    \qquad
    F=\alpha(r)^2Z(\beta(r)\psi) \cdot\psi^3.
  \end{equation*}
  We shall perform a Picard iteration in a suitably defined space.
  Let
  \begin{equation*}
    \textstyle
    a=\frac{2(m+1)}{m-1},
    \qquad
    a'=\frac{2(m+1)}{m+3},
    \qquad
    b=\frac{4m(m+1)}{2m^{2}-m-5}
    \quad
    (\iff \ \frac{m}{b}=\frac{m-1}{2}-\frac{1}{2a'})
  \end{equation*}
  and define the space $X$ of functions $u(t,x)$
  on $\mathbb{R}_{t}\times \mathbb{R}_{x}^{m}$,
  spherically symmetric in $x$, such that the following norm
  is finite:
  \begin{equation*}
    \|u(t,x)\|_{X}:=
    \|u\|_{L^{2a'}_{t}H^{\frac{n-1}{2}}_{b}}+
    \||D|^{\frac12}u\|_{L_{t}^{\infty}H^{\frac{n-1}2}}.
  \end{equation*}
  Note that the couple $(2a',b)$ satisfies the
  extended condition \eqref{eq:extadmWE} since
  $m\ge 5$, thus we can apply the
  Sobolev-Strichartz estimate \eqref{eq:sobstriWE}
  with the choice $(p,r)=(2a',b)$.
  In the following computations we use the notations
  \begin{equation*}
    \bra{D}=(1-\Delta)^{\frac12},
    \qquad
    \bra{\widetilde{D}}=(M-\Delta)^{\frac12}
  \end{equation*}
  where $M$ is chosen large enough (with respect to $V$)
  that we can apply Lemma \ref{sob1}. In a similar way,
  we write
  \begin{equation*}
    \bra{D_{V}}=(1-\Delta+V)^{\frac12},
    \qquad
    \bra{\widetilde{D_{V}}}=(M-\Delta+V)^{\frac12}.
  \end{equation*}

  We must estimate $\|\Lambda(\psi)\|_{X}$. The first term is
  \begin{equation*}
    \|\cos (t|D_V|)f\|_{X} =
    \|\bra{D}^{\frac{n-1}{2}}\cos(t|D_V|)f\|
             _{L^{2a'}L^b}+
    \||D|^{\frac12}\bra{D}^{\frac{n-1}{2}}\cos (t|D_V|)f\|
             _{L^\infty L^2}
  \end{equation*}
  which, by \eqref{eq:MH}, \eqref{eq1-sob1}
  and \eqref{eq1-lem2}, is equivalent to
  \begin{equation*}
    \simeq
    \|\cos(t|D_V|)\bra{\widetilde{D_{V}}}^{\frac{n-1}{2}}f\|
             _{L^{2a'}L^b}+
    \|D|^{\frac12}\cos (t|D_V|)|
             \bra{\widetilde{D_{V}}}^{\frac{n-1}{2}}f\|
             _{L^\infty L^2}.
  \end{equation*}
  Using the Strichartz-Sobolev estimate \eqref{eq:sobstriWE}
  for the first term,
  and directly \eqref{eq:striWE} for the second term,
  we obtain
  \begin{equation*}
    \lesssim
    \||D|^{\frac12}\bra{\widetilde{D_{V}}}^{\frac{n-1}{2}}f\|_{L^{2}}
    \simeq
    \||D|^{\frac12}\bra{D}^{\frac{n-1}{2}}f\|_{L^{2}}
  \end{equation*}
  where in the last step we used again
  \eqref{eq1-lem2} from Lemma \ref{sob2}
  and \eqref{eq:MH}.
  In a similar way, for the second term in $\Lambda(\psi)$
  we obtain
  \begin{equation*}
    \|\sin(t|D_V|)|D_V|^{-1}g\|_{X}
    \lesssim
    \||D|^{-\frac12}\bra{D}^{\frac{n-1}{2}}g\|_{L^{2}}.
  \end{equation*}
  Next, for the last term in $\Lambda(\psi)$ we can write,
  proceeding as before,
  \begin{equation*}
    \|\Box_V^{-1}F\|_X
    \simeq
    \|\Box_V^{-1}\bra{\widetilde{D_V}}^{\frac{n-1}{2}}F\|
          _{L^{2a'}L^b}+
    \||D|^{\frac12}\Box_V^{-1}
        \bra{\widetilde{D_V}}^{\frac{n-1}{2}}F\|
          _{L^\infty L^2}
  \end{equation*}
  and using \eqref{eq:strnhWE}
  (with the extension in Remark \ref{rem:sobstricest}),
  \eqref{eq1-sob1} and \eqref{eq:MH}
  \begin{equation*}
    \lesssim
    \|\bra{\widetilde{D_V}}^{\frac{n-1}{2}}F\|
          _{L^{a'}_{t,x}}
    \simeq
    \|\bra{\widetilde{D}}^{\frac{n-1}{2}}F\|
          _{L^{a'}_{t,x}}
    \simeq
    \|\bra{D}^{\frac{n-1}{2}}F\|
          _{L^{a'}_{t,x}}
  \end{equation*}
  since $(a,a)$ is an admissible couple satisfying
  \eqref{eq:admpq}. Summing up, we have proved
  \begin{equation}\label{eq:intstep}
    \|\Lambda(\psi)\|_{X}
    \lesssim
    \||D|^{\frac12}\bra{D}^{\frac{n-1}{2}}f\|_{L^{2}} +
    \||D|^{-\frac12}\bra{D}^{\frac{n-1}{2}}g\|_{L^{2}}+
    \|\bra{D}^{\frac{n-1}{2}}F\| _{L^{a'}_{t,x}}.
  \end{equation}

  It remains to estimate the nonlinear term
  $F=\alpha(r)^2Z(\beta(r)\psi) \cdot\psi^3$.
  We claim that
  \begin{equation}\label{eq:nlclaim}
    \textstyle
    \|\bra{D}^{\frac{n-1}{2}}F\|
        _{L_{t,x}^{a'}}\lesssim \|\psi\|_X^3
        \cdot \Phi_{1}(\|\psi\|_{X})
        \quad\text{with}\quad
        a'=\frac{2(m+1)}{m+3},\ m=2k+n
  \end{equation}
  for some continous function $\Phi_{1}(s)$.
  Define $\widetilde{Z}(r):=Z(r)-Z(0)$
  % \begin{equation*}
  %   \widetilde{Z}(r):=Z(r)-Z(0)
  % \end{equation*}
  and write
  \begin{equation}\label{eq:decompZ}
    \|\bra{D}^{\frac{n-1}{2}}F\|_{L^{a'}_{t,x}}=
    |Z(0)|\cdot
    \|\bra{D}^{\frac{n-1}{2}}(\alpha^{2}\psi^{3})\|_{L^{a'}_{t,x}}
    +
    \|\bra{D}^{\frac{n-1}{2}}(\widetilde{Z}(\beta \psi)
        \alpha^{2}\psi^{3}) \|_{L^{a'}_{t,x}}.
  \end{equation}
  For the first term in \eqref{eq:decompZ} we have,
  by the Kato-Ponce inequality \eqref{eq:poncekato2},
  \begin{equation*}
    \|\bra{D}^{\frac{n-1}{2}}(\alpha^{2}\psi^{3})\|_{L^{a'}_{t,x}}
    \lesssim
    \|\bra{D}^{\frac{n-1}{2}}\psi\|_{L^{\infty}L^{p_{1}}}
    \|\alpha\psi\|^{2}_{L^{2a'}L^{2p_{2}}}
    +
    \|\bra{D}^{\frac{n-1}{2}}(\alpha^{2}\psi^{2})\|
        _{L^{2a'}L^{p_{3}}}
    \|\psi\|_{L^{2a'}L^{p_{4}}}
  \end{equation*}
  where we have chosen
  \begin{equation*}
    \textstyle
    \frac{1}{p_{1}}=\frac{m-1}{2m},
    \qquad
    \frac{1}{p_{2}}=\frac{1}{a'}-\frac{m-1}{2m},
    \qquad
    \frac{1}{p_{3}}=\frac{2m+1}{2ma'}-\frac km,
    \qquad
    \frac{1}{p_{4}}=\frac km-\frac{1}{2ma'}.
  \end{equation*}
  By Sobolev embedding and by \eqref{eq:strauss} we have,
  since
  $k+\frac{n-1}{2}-\frac{m}{b}=\frac{1}{2a'}
  =1-\frac{m}{2p_{2}}$,
  \begin{equation*}
    \|v\|_{L^{p_{1}}}\lesssim\||D|^{\frac12}v\| _{L^{2}},
    \qquad
    \|\alpha v\|_{L^{2p_{2}}}\lesssim
    \|r^{k-1}v\|_{L^{2p_{2}}}\lesssim
    \||D|^{\frac{n-1}{2}}v\|_{L^{b}}
  \end{equation*}
  and these inequalities imply
  \begin{equation}\label{eq:I}
    \|\bra{D}^{\frac{n-1}{2}}\psi\|_{L^{\infty}L^{p_{1}}}
    \|\alpha\psi\|^{2}_{L^{2a'}L^{2p_{2}}}
    \lesssim
    \|\psi\|_{X}^{3}.
  \end{equation}
  Also by Sobolev embedding,
  since $\frac{m}{p_{4}}=k-\frac{1}{2a'}=\frac mb-\frac{n-1}{2}$,
  we have
  \begin{equation*}
    \|\psi\|_{L^{2a'} L^{p_{4}}}\lesssim
    \|\bra{D}^{\frac{n-1}{2}}\psi\|_{L^{2a'} L^{b}}
    \le\|\psi\|_{X}.
  \end{equation*}
  On the other hand, if we define
  \begin{equation*}
    \textstyle
    \frac{1}{p_{5}}=\frac{n+1}{2m}\equiv
    \frac{m+1}{2m}-\frac km
  \end{equation*}
  we have $\frac{1}{p_{3}}=\frac{1}{p_{5}}+\frac{1}{2p_{2}}$
  and by the Kato-Ponce inequality we can write
  \begin{equation*}
    \|\bra{D}^{\frac{n-1}{2}}(\alpha^{2}\psi^{2})\|
        _{L^{2a'} L^{p_{3}}}
    \lesssim
    \|\bra{D}^{\frac{n-1}{2}}(\alpha \psi)\|
        _{L^{\infty} L^{p_{5}}}
    \|\alpha \psi\|_{L^{2a'} L^{2p_{2}}};
  \end{equation*}
  the last factor is bounded by $\|\psi\|_{X}$
  as above, while for the other one we have
  \begin{equation*}
    \|\bra{D}^{\frac{n-1}{2}}(\alpha \psi)\|
        _{L^{\infty} L^{p_{5}}}
    \lesssim
    \||D|^{\frac12}\bra{D}^{\frac{n-1}{2}}\psi\|
        _{L^{\infty}L^{2}}\le\|\psi\|_{X};
  \end{equation*}
  in the last step we wrote
  $\|\bra{D}^{\frac{n-1}{2}}(\alpha \psi)\|_{L^{p_{5}}}
      \simeq
      \|\alpha \psi\|_{L^{p_{5}}}+
      \||D|^{\frac{n-1}{2}}\alpha \psi\|_{L^{p_{5}}}$
  and applied \eqref{eq:radialgamma} to each term.
  Summing up, we have proved that
  the first term in \eqref{eq:decompZ} can be estimated with
  \begin{equation}\label{eq:cubic}
    \|\bra{D}^{\frac{n-1}{2}}(\alpha^{2}\psi^{3})\|_{L^{a'}_{t,x}}
    \lesssim
    \|\psi\|_{X}^{3}.
  \end{equation}
  We now estimate the second term of
  \eqref{eq:decompZ}. By Kato-Ponce we have
  \begin{equation}\label{eq:II}
    \|\bra{D}^{\frac{n-1}{2}}[\widetilde{Z}\alpha^{2}\psi^{3}]\|
        _{L^{a'}_{t,x}}
    \lesssim
    \|\bra{D}^{\frac{n-1}{2}}(\alpha^{2}\psi^{3})\|
        _{L^{a'}_{t,x}}
    \|\widetilde{Z}\|_{L^{\infty}_{t,x}}
    +
    \|\alpha^{2}\psi^{3}\|_{L^{a'}L^{q_{1}}}
    \|\bra{D}^{\frac{n-1}{2}}\widetilde{Z}\|_{L^{\infty}L^{q_{2}}}
  \end{equation}
  where we choose
  \begin{equation*}
    \textstyle
    \frac{1}{q_{1}}=\frac{1}{a'}+\frac km-\frac{m-1}{2m}
    \equiv
    \frac{1}{a'}-\frac{n-1}{2m},
    \qquad\qquad
    \frac{1}{q_{2}}=\frac{1}{a'}-\frac{1}{q_{1}}
    \equiv
    \frac{m-1}{2m}-\frac km
    \equiv
    \frac{n-1}{2m}.
  \end{equation*}
  By \eqref{eq:strauss} we have
  \begin{equation*}
    \|\beta \psi\|_{L^{\infty}_{t,x}}
    \lesssim
    \|r^{k}\psi\|_{L^{\infty}}
    \lesssim
    \|\psi\|_{L^{\infty}\dot H^{\frac m2-k}}
    =
    \|\psi\|_{L^{\infty}\dot H^{\frac n2}}
    \le
    \|\psi\|_{X}
  \end{equation*}
  and this implies, for some continuous $\Phi_{1}(s)$,
  \begin{equation*}
    \|\widetilde{Z}(\beta \psi)\|_{L^{\infty}_{t,x}}
    \lesssim
    \Phi_{1}(\|\psi\|_{X}).
  \end{equation*}
  Further, by Lemma \ref{lem:coif},
  we have
  \begin{equation*}
    \|\bra{D}^{\frac{n-1}{2}}\widetilde{Z}(\beta \psi)\|
        _{L^{\infty}L^{q_{2}}}
    \lesssim
    \Phi_{2}(\|\beta \psi\|_{L^{\infty}_{t,x}})
    \|\bra{D}^{\frac{n-1}{2}}(\beta \psi)\|
        _{L^{\infty}L^{q_{2}}}
  \end{equation*}
  while by \eqref{eq:radialgamma} (writing as above
  the nonhomogeneous Sobolev norm as a sum of homogeneous
  terms and applying \eqref{eq:radialgamma} to each term)
  \begin{equation*}
    \|\bra{D}^{\frac{n-1}{2}}(\beta \psi)\|
        _{L^{\infty}L^{q_{2}}}
    \lesssim
    \||D|^{\frac12}\bra{D}^{\frac{n-1}{2}}\psi\|_{L^{\infty}L^{2}}
    \le
    \|\psi\|_{X}
  \end{equation*}
  which gives
  \begin{equation*}
    \|\bra{D}^{\frac{n-1}{2}}\widetilde{Z}(\beta \psi)\|
        _{L^{\infty}L^{q_{2}}}
    \lesssim
    \Phi_{2}(\|\psi\|_{X})\|\psi\|_{X}.
  \end{equation*}
  Finally, we have by Sobolev embedding
  \begin{equation*}
    \|\alpha^{2}\psi^{3}\|_{L^{a'}L^{q_{1}}}
    \lesssim
    \|\bra{D}^{\frac{n-1}{2}}(\alpha^{2}\psi^{3})\|
        _{L^{a'}_{t,x}}
  \end{equation*}
  and coming back to \eqref{eq:II} we obtain, recalling also
  \eqref{eq:cubic},
  \begin{equation*}
    \|\bra{D}^{\frac{n-1}{2}}[\widetilde{Z}\alpha^{2}\psi^{3}]\|
        _{L^{a'}_{t,x}}
    \lesssim
    \|\psi\|_{X}^{3}\cdot
    \Phi_{1}(\|\psi\|_{X})+
    \|\psi\|_{X}^{4}\cdot
    \Phi_{2}(\|\psi\|_{X}).
  \end{equation*}
  Putting everything together, we obtain the claim
  \eqref{eq:claim} and in conclusion we have proved
  \begin{equation*}
      \|\Lambda(\psi)\|_{X}
      \lesssim
      \||D|^{\frac12}\bra{D}^{\frac{n-1}{2}}f\|_{L^{2}} +
      \||D|^{-\frac12}\bra{D}^{\frac{n-1}{2}}g\|_{L^{2}}+
      \|\psi\|_X^3 \cdot \Phi(\|\psi\|_{X})
  \end{equation*}
  In a similar way we can prove the estimate
  \begin{equation*}
    \|\Lambda(\psi_{1})-\Lambda(\psi_{2})\|_{X}
    \lesssim
    \|\psi_{1}-\psi_{2}\|_{X}
    \cdot
    [\|\psi_{1}\|_{X}^{2}+\|\psi_{2}\|_{X}^{2}]
    \cdot
    \Phi_{3}(\|\psi_{1}\|_{X}+\|\psi_{2}\|_{X})
  \end{equation*}
  and this is sufficient to deduce the existence of a fixed
  point for $\Lambda$ in $X$, provided the quantity
  \begin{equation*}
    \||D|^{\frac12}\bra{D}^{\frac{n-1}{2}}f\|_{L^{2}} +
    \||D|^{-\frac12}\bra{D}^{\frac{n-1}{2}}g\|_{L^{2}}
  \end{equation*}
  is sufficiently small.
  Note that continuity in time
  of the solution follows from the fact that
  the fixed point can be obtained as the limit in the space
  $X$ of a sequence of smooth Picard iterates.
  This concludes the proof of the Theorem.
\end{proof}

\begin{remark}[]\label{rem:LWPlargedata}
  We show how to modify the previous proof in order to deduce
  local existence with large data, as mentioned in
  Remark \ref{rem:LWP}.
  Let $T>0$ and define a version $X_{T}$ of the space $X$
  in which the norms $L^{2a'}_{t}H^{\frac{n-1}{2}}_{b}$ and
  $L_{t}^{\infty}H^{\frac{n-1}2}$ are now restricted to
  the time interval $t\in[0,T]$; we denote the restricted
  norms with $L^{2a'}_{T}H^{\frac{n-1}{2}}_{b}$ and
  $L_{T}^{\infty}H^{\frac{n-1}2}$ respectively.
  The operator $\Lambda$ is defined as above, and we look for a fixed point
  in the closed ball of $X_{T}$
  % \begin{equation*}
  %   B_{\epsilon}=\{\psi\in X_{T}:
  %   \|\psi\|_{L^{2a'}_{T}H^{\frac{n-1}{2}}_{b}}+
  %       \||D|^{\frac12}(\psi-\psi_{lin})\|_{L_{T}^{\infty}H^{\frac{n-1}2}}\le
  %       \epsilon \}
  % \end{equation*}
  \begin{equation*}
    B_{\epsilon}=\{\psi\in X_{T}:
    \|\psi-\psi_{lin}\|_{L^{2a'}_{T}H^{\frac{n-1}{2}}_{b}}+
        \||D|^{\frac12}(\psi-\psi_{lin})\|_{L_{T}^{\infty}H^{\frac{n-1}2}}\le
        \epsilon \}
  \end{equation*}
  where
  \begin{equation*}
    \psi_{lin}:=
    \cos (t|D_V|)f+
    \sin(t|D_V|)|D_V|^{-1}g.
  \end{equation*}
  Note that if $\psi\in B_{\epsilon}$ we have
  \begin{equation*}
    \||D|^{\frac12}\psi\|_{L^{\infty}_{T}H^{\frac{n-1}{2}}}\le
    \epsilon+E_{0}
  \end{equation*}
  where
  \begin{equation}\label{eq:X2Y}
    E_{0}:=
    \||D|^{\frac12}f\|_{H^{\frac{n-1}2}}+
    \||D|^{-\frac12}g\|_{H^{\frac{n-1}2}},
  \end{equation}
  while
  \begin{equation*}
    \|\psi\|_{L^{2a'}_{T}H^{\frac{n-1}{2}}_{b}}\le 2 \epsilon
  \end{equation*}
  provided $T$ is sufficiently small, since the
  $L^{2a'}_{T}H^{\frac{n-1}{2}}_{b}$ norm of $\psi_{lin}$ is bounded
  and hence tends to 0 as $T\to0$.
  We have now, for $\psi\in B_{\epsilon}$,
  \begin{equation*}
    \|\Lambda(\psi)-\psi_{lin}\|_{X_{T}} =
    % \|u_{lin}\|_{L^{2a'}_{T}H^{\frac{n-1}{2}}_{b}}+
    \|\Box^{-1}_{V}F\|_{L^{2a'}_{T}H^{\frac{n-1}{2}}_{b}}+
    \|\Box^{-1}_{V}F\|_{L_{T}^{\infty}H^{\frac{n-1}2}}.
  \end{equation*}
  Repeating the steps of the previous proof, but using now time
  localized versions of the Strichartz estimates, we prove that
  \begin{equation*}
    \|\Lambda(\psi)-\psi_{lin}\|_{X_{T}}
    \lesssim
    % \|u_{lin}\|_{L^{2a'}_{T}H^{\frac{n-1}{2}}_{b}}+
    \|\psi\|_{L^{2a'}_{T}H^{\frac{n-1}{2}}_{b}}^{2}
    \||D|^{\frac12}\psi\|_{L^{\infty}_{T}H^{\frac{n-1}{2}}}
    \cdot \Phi(\|\psi\|_{X_{T}})
    \le \epsilon/2
  \end{equation*}
  provided $\epsilon$ is chosen small enough; in particular,
  $\Lambda$ takes $B_{\epsilon}$ into itself.
  A corresponding estimate can be proved for
  $\Lambda(\psi_{1})-\Lambda(\psi_{2})$,
  with $\psi_{1},\psi_{2}\in B_{\epsilon}$:
  \begin{equation*}
    \|\Lambda(\psi_{1})-\Lambda(\psi_{2})\|_{X_{T}}\lesssim
    [\|\psi_{1}\|_{L^{2a'}_{T}H^{\frac{n-1}{2}}_{b}}+
      \|\psi_{2}\|_{L^{2a'}_{T}H^{\frac{n-1}{2}}_{b}}] \cdot
    \Phi_{3}(\|\psi_{1}\|_{X_{T}}+\|\psi_{2}\|_{X_{T}})
    \cdot
    \|\psi_{1}-\psi_{2}\|_{X_{T}};
  \end{equation*}
  since the factor in square brackets is bounded by $4 \epsilon$,
  we see that
  $\Lambda$ is a contraction if $\epsilon$ is sufficiently small.

  An inspection of the previous argument shows that the
  small data assumption in the main Theorem can be weakened,
  by assuming only that the linear part of
  the flow $u_{lin}$ is sufficiently small.
\end{remark}

It is possible to improve the uniqueness part
of the previous result by a slight increase of
the regularity of the initial data:

\begin{theorem}[Regularity and unconditional uniqueness]
  \label{the:regulUU}
  Consider Problem \eqref{eq:cauchykey}, \eqref{eq:cauchykeydata}
  under the same assumptions of Theorem \ref{the:keyWE}.

  In the case $h_{\infty}>0$, if
  for some $0\le\delta<k$ the quantity
  $\|f\|_{H^{\frac{n}{2}+\delta}}
     +\|g\|_{H^{\frac{n}{2}-1+\delta}}$
  is sufficiently small,
  then the problem has a global solution $\psi$ with
  $\psi\in L^{\infty}H^{\frac{n}{2}+\delta}
    \cap CH^{\frac{n}{2}+\delta}
    \cap L^{p}H^{\frac{n-1}{2}}_{q}$,
  with $p,q$ as in Theorem \ref{the:keyWE}.
  If $\delta\ge \frac{1}{m+1}$, this is the unique solution
  in $CH^{\frac{n}{2}+\delta}$.

  In the case $h_{\infty}=0$, if
  for some $0\le\delta<k$ the quantity
  $\||D|^{\frac12}f\|_{H^{\frac{n-1}{2}+\delta}}
     +\||D|^{-\frac12}g\|_{H^{\frac{n-1}{2}+\delta}}$
  is sufficiently small,
  then the problem has a unique global solution $\psi$ with
  $|D|^{\frac12}\psi\in L^{\infty}H^{\frac{n-1}{2}+\delta}
    \cap CH^{\frac{n-1}{2}+\delta}$ and
  $\psi\in L^{p}H^{\frac{n-1}{2}}_{q}$.
  If $\delta\ge \frac{1}{m+1}$, this is the unique solution
  with $|D|^{\frac12}\psi\in CH^{\frac{n-1}{2}+\delta}$.

\end{theorem}

\begin{proof}%[of ...]
  As usual we give the detail of the proof only in
  the case $h_{\infty}=0$ which is more delicate.
  Let $X_{\delta}$ be the space with norm
  ($n=m-2k$)
  \begin{equation*}
    \|u(t,x)\|_{X_{\delta}}:=
    \|u\|_{L^{2a'}_{t}H^{\frac{n-1}{2}+\delta}_{b}}+
    \||D|^{\frac12}u\|_{L_{t}^{\infty}H^{\frac{n-1}2+\delta}},
    \qquad
    \delta\ge0
  \end{equation*}
  so that the space used in the previous proof is $X_{0}$.
  Following the same steps we arrive at the estimate
  \begin{equation*}
    \|\Lambda(\psi)\|_{X_{\delta}}\lesssim
    \||D|^{\frac12}\bra{D}^{\frac{n-1}{2}_{\delta}}f\|_{L^{2}} +
    \||D|^{-\frac12}\bra{D}^{\frac{n-1}{2}+\delta}g\|_{L^{2}}+
    \|\bra{D}^{\frac{n-1}{2}+\delta}F\| _{L^{a'}_{t,x}}.
  \end{equation*}
  with $F=\alpha^{2}Z(\beta \psi)\cdot \psi^{3}$. The proof
  of the nonlinear estimate proceeds as before; instead of
  \eqref{eq:I} we get
  \begin{equation*}
    \|\bra{D}^{\frac{n-1}{2}+\delta}\psi\|_{L^{\infty}L^{p_{1}}}
    \|\alpha\psi\|^{2}_{L^{2a'}L^{2p_{2}}}
    \lesssim
    \|\psi\|_{X_{\delta}}\|\psi\|_{X_{0}}^{2}
    \lesssim
    \|\psi\|_{X_{\delta}}^{3}.
  \end{equation*}
  with the same choice of indices, since
  $\|\psi\|_{X_{0}}\le\|\psi\|_{X_{\delta}}$. In a similar
  way we have, with the same indices as before,
  \begin{equation*}
    \|\bra{D}^{\frac{n-1}{2}+\delta}(\alpha^{2}\psi^{2})\|
      _{L^{2a'}L^{p_{3}}}
    \lesssim
    \|\bra{D}^{\frac{n-1}{2}+\delta}(\alpha \psi)\|
      _{L^{\infty}L^{p_{5}}}
    \|\alpha \psi\|_{L^{2a'}L^{2p_{2}}}
  \end{equation*}
  where the last factor is bounded by $\|\psi\|_{X_{0}}$;
  on the other hand,
  \begin{equation*}
    \|\bra{D}^{\frac{n-1}{2}+\delta}(\alpha \psi)\|
      _{L^{\infty}L^{p_{5}}}
    \lesssim
    \|\alpha \psi\|_{L^{\infty}L^{p_{5}}}
    +
    \||D|^{\frac{n-1}{2}+\delta}(\alpha \psi)\|
      _{L^{\infty}L^{p_{5}}}
  \end{equation*}
  where
  $\|\alpha \psi\|_{L^{\infty}L^{p_{5}}}\lesssim\|\psi\|_{X_{0}}$
  while, using Lemma \ref{lem:radial33},
  \begin{equation*}
    \||D|^{\frac{n-1}{2}+\delta}(\alpha \psi)\|
      _{L^{\infty}L^{p_{5}}}
    \lesssim
    \||D|^{\frac n2+\delta}\psi\|_{L^{\infty}L^{2}}
  \end{equation*}
  provided
  \begin{equation*}
    \textstyle
    \frac12-\frac1{p_{5}}\le
    \frac12<
    \frac m2-\frac{n-1}{2}-\delta
    \equiv
    k+\frac12-\delta
  \end{equation*}
  i.e., $\delta<k$. In conclusion we have
  \begin{equation*}
    \|\bra{D}^{\frac{n-1}{2}+\delta}(\alpha^{2}\psi^{3})\|
      _{L^{a'}_{t,x}}
    \lesssim
    \|\psi\|^{3}_{X_{\delta}}
    \quad\text{provided}\quad
    \delta<k.
  \end{equation*}
  The estimate of the full nonlinear term
  $F=\alpha^{2}\psi^{3}Z(\beta \psi)$ is similar.
  Thus we obtain global existence and uniqueness in $X_{\delta}$
  for all $0\le \delta<k$. This proves the regularity part
  of the statement.

  To prove unconditional uniqueness, consider two solutions
  $\psi,\widetilde{\psi}$
  belonging to $C([0,T];H^{\frac{n}{2}+\delta})$ for
  some $T>0$; we shall prove that if $\delta\ge\frac{1}{m+1}$
  then $\psi \equiv \widetilde{\psi}$ on some smaller interval
  $t\in[0,\epsilon]$, and this will conclude the proof.
  The difference $\chi=\widetilde{\psi}-\psi$ satisfies the equation
  \begin{equation*}
   \chi''-\Delta \chi+V \chi=F_{1}+F_{2},
   \qquad
   \chi(0,r)=\chi_{t}(0,r)=0
  \end{equation*}
  where
  \begin{equation*}
    F_{1}=\alpha^{2}(\chi^{3}+3 \chi^{2}\psi+3 \chi \psi^{2})
    Z(\beta \widetilde{\psi})
  \end{equation*}
  and
  \begin{equation*}
    \textstyle
    F_{2}=\alpha^{2} \psi^{3}
    [Z(\beta \widetilde{\psi})-Z(\beta \psi)]=
    \alpha^{2}\psi^{3}
    \beta \chi \cdot\int_{0}^{1}
        Z((1-s)\beta \widetilde{\psi}+s \beta \psi)ds.
  \end{equation*}
  Note that
  \begin{equation*}
    \textstyle
    Z_{0}(t,r):=Z(\beta \widetilde{\psi}),
    \qquad
    Z_{1}(t,r):=
    \int_{0}^{1} Z((1-s)\beta \widetilde{\psi}+s \beta \psi)ds
  \end{equation*}
  are bounded functions by estimate \eqref{eq:strauss}.
  We now apply the Strichartz estimate \eqref{eq:strnhWE} for the special
  case $p=q=\widetilde{p}=\widetilde{q}$:
  \begin{equation}\label{eq:strichpt}
    \textstyle
    \|\int_{0}^{t}H^{-\frac12}e^{i(t-s)\sqrt{H}}Fds\|
        _{L^{\frac{2(m+1)}{m-1}}_{t,x}}
    \lesssim
    \|F\|_{L^{\frac{2(m+1)}{m+3}}_{t,x}}.
  \end{equation}
  Localizing the estimate on a time interval $I=[0,\epsilon]$
  to be chosen, we obtain
  \begin{equation*}
    \|\chi\|_{L^{\frac{2(m+1)}{m-1}}_{t\in I,x}}
    \lesssim
    \|F_{1}+F_{2}\|_{L^{\frac{2(m+1)}{m+3}}_{t\in I,x}}.
  \end{equation*}
  By H\"{o}lder's inequality and \eqref{eq:strauss}
  we have
  \begin{equation*}
    \|\chi \cdot(\alpha\chi)^{2}\cdot Z_{0}\|
        _{L^{\frac{2(m+1)}{m+3}}}
    \lesssim
    \|\chi\|_{L^{\frac{2(m+1)}{m-1}}}
    \|\alpha \chi\|^{2}_{L^{m+1}}
    \lesssim
    \|\chi\|_{L^{\frac{2(m+1)}{m-1}}}
    \|\chi\|^{2}
        _{L^{m+1}_{t\in I}\dot H^{\frac n2+\frac{1}{m+1}}}.
  \end{equation*}
  Similar estimates hold for the other two terms in $F_{1}$:
  \begin{equation*}
    \|(\alpha\chi) (\alpha \psi)\chi \cdot Z_{0}\|
        _{L^{\frac{2(m+1)}{m+3}}}
    \lesssim
    \|\chi\|_{L^{\frac{2(m+1)}{m-1}}}
    \|\psi\|
        _{L^{m+1}_{t\in I}\dot H^{\frac n2+\frac{1}{m+1}}}
    \|\chi\|
        _{L^{m+1}_{t\in I}\dot H^{\frac n2+\frac{1}{m+1}}}
  \end{equation*}
  and
  \begin{equation*}
    \|\chi \cdot(\alpha\psi)^{2}\cdot Z_{0}\|
        _{L^{\frac{2(m+1)}{m+3}}}
    \lesssim
    \|\chi\|_{L^{\frac{2(m+1)}{m-1}}}
    \|\psi\|^{2}
        _{L^{m+1}_{t\in I}\dot H^{\frac n2+\frac{1}{m+1}}}.
  \end{equation*}
  and in conclusion
  \begin{equation*}
    \|F_{1}\|_{L^{\frac{2(m+1)}{m+3}}}
    \lesssim
    \|\chi\|_{L^{\frac{2(m+1)}{m-1}}}
    \cdot
    (\|\psi\|
        _{L^{m+1}_{t\in I}\dot H^{\frac n2+\frac{1}{m+1}}}
    +
    \|\chi\|
        _{L^{m+1}_{t\in I}\dot H^{\frac n2+\frac{1}{m+1}}})^{2}.
  \end{equation*}
  To estimate $F_{2}$ we write analogously
  \begin{equation*}
    \|(\alpha \psi)^{2}\chi \cdot (\beta \psi)Z_{1}\|
        _{L^{\frac{2(m+1)}{m+3}}}
    \lesssim
    \|\chi\|_{L^{\frac{2(m+1)}{m-1}}}
    \|\psi\|^{2}
        _{L^{m+1}_{t\in I}\dot H^{\frac n2+\frac{1}{m+1}}}.
  \end{equation*}
  Summing up, the function $\chi$ satisfies the following
  estimate on the interval $t\in I=[0,\epsilon]$:
  \begin{equation*}
    \|\chi\|_{L^{\frac{2(m+1)}{m-1}}_{t\in I,x}}
    \lesssim
    (\|\psi\|
        _{L^{m+1}_{t\in I}\dot H^{\frac n2+\frac{1}{m+1}}}
    +
    \|\chi\|
        _{L^{m+1}_{t\in I}\dot H^{\frac n2+\frac{1}{m+1}}})^{2}
    \cdot
    \|\chi\|_{L^{\frac{2(m+1)}{m-1}}_{t\in I,x}}
  \end{equation*}
  Since we know a priori that the
  $L^{\infty}(I;H^{\frac n2+\frac{1}{m+1}})$ norm of
  $\chi$ and $\psi$ is bounded, we deduce
  \begin{equation*}
    \|\psi\|
        _{L^{m+1}_{t\in I}\dot H^{\frac n2+\frac{1}{m+1}}}
    +
    \|\chi\|
        _{L^{m+1}_{t\in I}\dot H^{\frac n2+\frac{1}{m+1}}}
    \to0
    \quad\text{as}\quad
    \epsilon\to0.
  \end{equation*}
  Hence for $\epsilon$ sufficiently small we obtain
  $\|\chi\|\le \frac12\|\chi\|$ which implies $\chi \equiv0$
  for $t\in[0,\epsilon]$, as claimed.
\end{proof}

\section{The equivariant wave map equation}
\label{sec:applications}

This final section of the paper contains the main application of
Theorem \ref{the:keyWE} to the equivariant wave map
equation. The assumptions on the base manifolds
define a class of manifolds which for the sake of
exposition we call \emph{admissible}.
Note that if a smooth and rotationally
symmetric manifold $M^{n}$ has a global metric
$dr^{2}+h(r)^{2}d \omega^{2}_{\mathbb{S}^{n-1}}$,
then $h(r)$ must be the restriction to $\mathbb{R}^{+}$
of a $C^{\infty}$ odd function, with $h(0)=0$ and $h'(0)=1$.
Note also that our result applies
to $C^{k}$ manifolds with $k=[\frac{n-1}{2}]+3$, but
for simplicity we confine ourselves to the smooth case.

\begin{definition}[Admissible manifold]\label{def:admman}
  We say that a function $h:\mathbb{R}^{+}\to \mathbb{R}^{+}$
  is \emph{admissible} if it is the restriction of a smooth
  odd function with $h'(0)=1$ and in addition:
  \begin{enumerate}
  [noitemsep,topsep=0pt,parsep=0pt,partopsep=0pt,
  label=\textit{(\roman*)}]
    % \item $h'(0)=1$, $h\ge cr$ and
    % $h$ is the restriction of an odd smooth function
    % on $\mathbb{R}$;
    \item There exists $h_{\infty}\ge0$ such
    that
    $H(r):=h^{\frac{1-n}{2}}(h^{\frac{n-1}{2}})''=
      h_{\infty}+O(r^{-2})$ for $r\gg1$.
    \item $H^{(j)}(r)=O(r^{-1})$ and
    $(h^{-\frac12})^{(j)}=O(r^{-\frac12-j})$
    for $r\gg1$ and $1\le j\le[\frac{n-1}{2}]$.
    \item There exist $c,\delta_{0}>0$ such that
    for $r>0$ we have $h(r)\ge cr$ while the function
    $P(r)=rH(r)-rh_{\infty}+\frac{1-\delta_{0}}{4r}$
    satisfies the condition
    $P(r)\ge 0 \ge P'(r)$.
    % \begin{equation}\label{eq:condPimp}
    %   P(r)\ge 0 \ge P'(r).
    % \end{equation}
  \end{enumerate}
  We say that a manifold $M^{n}$ is \emph{admissible} if it
  has a global metric of the form
  $dr^{2}+h(r)^{2}d \omega^{2}_{\mathbb{S}^{n-1}}$
  with $h(r)$ admissible.
\end{definition}

The simplest admissible manifolds
are the flat space, with $h(r)=r$ and $h_{\infty}=0$,
and the real hyperbolic space $\mathbb{H}^{n}$, with
$h(r)=\sinh(r)$ and $h_{\infty}>0$. However the class
is substantially larger, and we shall exhibit
a few interesting examples below.

Consider now the equivariant wave map equation
\begin{equation}\label{eq:radialWMbis}
  \phi_{tt}-\phi_{rr}-(n-1)\frac{h'(r)}{h(r)}\phi_{r}+
  k(n-2+k)
  \frac{g(\phi)g'(\phi)}{h(r)^{2}}=0,\qquad
\end{equation}
with initial data
\begin{equation}\label{eq:radialWMdatabis}
  \phi(0,x)=\phi_{0}(x),\qquad
  \partial_{t}\phi(0,x)=\phi_{1}(x)
\end{equation}
from the admissible, $n$-dimensional base manifold $M^{n}$
with metric
$dr^{2}+h(r)^{2}d \omega^{2}_{\mathbb{S}^{n-1}}$
to a target, $\ell$-dimensional
manifold $N^{\ell}$, with metric
$d\rho^{2}+g(\rho)^{2}d \omega^{2}_{\mathbb{S}^{\ell-1}}$.
Note that the
base manifold is noncompact while the target can
arbitrary, thus
$g:[0,A)\to \mathbb{R}^{+}$ with $A$ finite or infinite.

In the following statement we use the
notation $|D_{M}|=(-\Delta_{M})^{\frac12}$.
The Sobolev space $H^{s}$ on $M^{n}$
is defined through the norm
\begin{equation*}
  \|\phi\|_{H^{s}}:=\|(1-\Delta_{M})^{\frac s2}v\|_{L^{2}(M^{n})}
\end{equation*}
while the weighted Sobolev spaces $H^{s}_{q}(w)$
on $M^{n}$ are defined through the norms
\begin{equation*}
  \|v\|_{H^{s}_{q}(w)}:=
  \|w^{-1}v\|_{H^{s}_{q}(\mathbb{R}^{m})},
  \qquad
  w(r):=\frac{r^{\frac{m-1}{2}}}{h(r)^{\frac{n-1}{2}}}.
\end{equation*}

Then we have:

\begin{corollary}[]\label{cor:hrfirst}
  Let $n\ge3$, $k\ge1$ and $0\le \delta<k$.
  Let $M^{n}$ and $N^{\ell}$ be two rotationally invariant
  manifolds of dimension $n$ and $\ell$ respectively,
  with $M^{n}$ admissible, and let $h_{\infty}$ be as
  in Definition \ref{def:admman}.

  If $h_{\infty}>0$ and
  $\| \phi_{0}\|_{H^{\frac{n}{2}+\delta}}+
    \|\phi_{1}\|_{H^{\frac{n}{2}-1+\delta}}$
  is sufficiently small, Problem \eqref{eq:radialWMbis},
  \eqref{eq:radialWMdatabis} has a unique global solution
  $\phi\in L^{\infty}H^{\frac{n}{2}+\delta}
    \cap CH^{\frac{n}{2}+\delta}
    \cap L^{p}H^{\frac{n-1}{2}+\delta}_{q}(w)$,
  with $p,q$ as in Theorem \ref{the:keyWE}.
  Moreover, if
  $\delta\ge \frac{1}{m+1}$, this is the unique solution
  in $CH^{\frac{n}{2}+\delta}$.

  If $h_{\infty}=0$ and
  $\||D_{M}|^{\frac12} \phi_{0}\|_{H^{\frac{n-1}{2}+\delta}(M)}+
    \||D_{M}|^{-\frac12}\phi_{1}\|_{H^{\frac{n-1}{2}+\delta}(M)}$
  is sufficiently small, Problem \eqref{eq:radialWMbis},
  \eqref{eq:radialWMdatabis} has a unique global solution
  $\phi$ with
  $|D_{M}|^{\frac12}\phi\in L^{\infty}H^{\frac{n-1}{2}+\delta}
    \cap CH^{\frac{n-1}{2}+\delta}$
  and
  $\phi\in L^{p}H^{\frac{n-1}{2}+\delta}_{q}(w)$.
  Moreover, if
  $\delta\ge \frac{1}{m+1}$, this is the unique solution
  with $|D_{M}|^{\frac12}\phi\in CH^{\frac{n-1}{2}+\delta}$.
\end{corollary}

\begin{proof}%[of ...]
  The proof is just a transposition of Theorems \ref{the:keyWE}
  and \ref{the:regulUU} via the change of variables
  $\phi(t,r)=w(r)\psi(t,r)$,
  using the equivalence of norms given by
  Lemma \ref{lem:equivLBnorm2}.
\end{proof}

We conclude the paper with a discussion of the class of
admissible manifolds. It is not difficult to come up with
explicit examples, notably
the flat space $\mathbb{R}^{n}$,
the real hyperbolic spaces $\mathbb{H}^{n}$ for $n\ge3$,
and some spaces with polynomial growth of the metric;
this already shows that admissibility does not
impose a constraint on the growth rate of the metric
at infinity.
The examples are discussed in detail below.

However, we first give some stability criteria
which show that manifolds sufficiently
close to an admissible manifold are also admissible.
A simple criterion is the following:

\begin{proposition}[]\label{pro:crit1}
  Let $h,h_{\epsilon}$ be restrictions to $\mathbb{R}^{+}$
  of smooth odd functions, with
  $h'(0)=h_{\epsilon}'(0)=1$, and let
  $H(r):=h^{\frac{1-n}{2}}(h^{\frac{n-1}{2}})''$ and
  $H_{\epsilon}(r):=h_{\epsilon}^{\frac{1-n}{2}}
      (h_{\epsilon}^{\frac{n-1}{2}})''$.
  Assume the following conditions:
  \begin{enumerate}
  [noitemsep,topsep=0pt,parsep=0pt,partopsep=0pt]
    \item $h_{\epsilon}>cr$ for some $c>0$ and all $r>0$;
    \item $|H-H_{\epsilon}|\le \frac{\epsilon}{r^{2}}$
    for some $\epsilon>0$ and all $r>0$;
    \item $|H'-H_{\epsilon}'|\le \frac{\epsilon}{r^{3}}$
    for some $\epsilon>0$ and all $r>0$;
    \item $H^{(j)}-H_{\epsilon}^{(j)}=O(r^{-1})$
    for $r\gg1$, $j\le[\frac{n-1}{2}]$;
    \item $(h_{\epsilon}^{-\frac12})^{(j)}=O(r^{-j-\frac12})$
    for $r\gg1$, $j\le[\frac{n-1}{2}]$.
  \end{enumerate}
  If $h$ is admissible, then $h_{\epsilon}$ is also admissible,
  provided $\epsilon$ is sufficiently small.
\end{proposition}

\begin{proof}%[of ...]
  The criterion is a restatement of
  Definition \ref{def:admman}.
  Note in particular that, by (2), the limit $h_{\infty}$
  at infinity of $H$ and $H_{\epsilon}$
  is the same, and if we define
  $P_{\epsilon}=r(H_{\epsilon}-h_{\infty})+\frac{1-\delta_{1}}{4r}$
  with $0<\delta_{1}<\delta_{0}$, the property
  $P_{\epsilon}\ge0\ge P'_{\epsilon}$ follows immediately
  from the corresponding property for $P$
  and assumptions (2), (3), provided
  $\epsilon$ is small enough.
\end{proof}

We can make the criterion easier to apply by introducing the
functions
\begin{equation}\label{eq:psig}
  \sigma_{j}(r):=\frac{h_{\epsilon}^{(j)}}{h_{\epsilon}},
  \qquad
  p_{j}(r):=\frac{h_{\epsilon}^{(j)}-h^{(j)}}{h}.
\end{equation}
Then the difference $H_{\epsilon}-H$ can be expressed as
\begin{equation}\label{eq:diffHH}
  \textstyle
  H_{\epsilon}-H=
  \frac{n-1}{2}(p_{2}-\sigma_{2}p_{0})
  +
  \frac{(n-1)(n-3)}{4}
  (2 \sigma_{1}+\sigma_{1}p_{0}-p_{1})(p_{1}-\sigma_{1}p_{0})
\end{equation}
while by Faa' di Bruno's formula we have
\begin{equation}\label{eq:root}
  (h_{\epsilon}^{-\frac12})^{(j)}=
  \sum_{\nu=0}^{j}\sum_{j_{1}+\dots+j_{\nu}=j}
  C
  h_{\epsilon}^{-\frac12}
  \sigma_{j_{1}}\cdots \sigma_{j_{\nu}}
\end{equation}
where the constants $C=C(\nu,j_{1},\dots,j_{\nu})$
may be different for each term of the sum. Note also
the recursive relations
\begin{equation}\label{eq:recurs}
  \sigma_{j}'=
    \sigma_{j+1}-\sigma_{j}\sigma_{1}
  \qquad
  p_{j}'=
    p_{j+1}+
    p_{j}(p_{1}-\sigma_{1}p_{0}-\sigma_{1}).
\end{equation}
We can now give an effective criterion which is useful
in case the metric $h(r)$ grows exponentially, so that
we can not expect any decay at infinity for $\sigma_{j}$.
The following conditions are not sharp
but easy to check on concrete
examples:

\begin{proposition}[]\label{pro:crit2}
  Let $h,h_{\epsilon}:\mathbb{R}^{+}\to \mathbb{R}^{+}$
  be restrictions of smooth odd functions with
  $h'(0)=h_{\epsilon}'(0)=1$.
  Assume that
  \begin{enumerate}
  [noitemsep,topsep=0pt,parsep=0pt,partopsep=0pt,
  label=\textit{(\roman*)}]
    \item $h_{\epsilon}(r)\gtrsim r+r^{n}$ for $r>0$
    \item $|h_{\epsilon}^{(j)}|\lesssim h_{\epsilon}$
      for $j\le[\frac{n-1}{2}]+2$ and $r\gg1$
    \item $|h^{(j)}-h_{\epsilon}^{(j)}|/h\le
      \epsilon \bra{r}^{-3}$ for $j\le3$ and $r>0$
    \item $|h^{(j)}-h_{\epsilon}^{(j)}|/h\lesssim r^{-1}$
      for $j\le[\frac{n-1}{2}]+2$ and $r\gg1$.
  \end{enumerate}
  If $h$ is admissible, then
  $h_{\epsilon}$ is also admissible, provded $\epsilon$
  is small enough.
\end{proposition}

\begin{proof}%[of ...]
  We use the notations \eqref{eq:psig}.
  To prove the claim it is sufficient to check conditions
  (1)--(5) in Proposition \ref{pro:crit1}.

  Applying (ii) for large $r$, and the definition
  and smoothness of $h_{\epsilon}$ near zero, we have
  \begin{equation*}
    |\sigma_{j}|\lesssim \frac1r+1
  \end{equation*}
  for all $r>0$ and $j\le[\frac{n-1}{2}]+2$.
  We have also by (iii)
  \begin{equation*}
    |p_{j}|\le \epsilon \bra{r}^{-3}
    \quad\text{for}\quad r>0,\ j\le3.
  \end{equation*}
  Then by \eqref{eq:diffHH} we obtain easily
  \begin{equation*}
    |H_{\epsilon}-H| \lesssim \epsilon r^{-2},
    \qquad
    |H_{\epsilon}'-H'|\lesssim \epsilon r^{-3}
  \end{equation*}
  which are conditions (2), (3), while (1) is implied by (i).
  Moreover, using the recursions \eqref{eq:recurs} and
  assumption (iv), we see
  that $\sigma_{j}$ is bounded for large $r$
  and $j\le[\frac{n-1}{2}]+2$; this implies that
  condition (4) is satisfied. Finally, recalling
  \eqref{eq:root} and using (i), we have for large $r$
  \begin{equation*}
    |(h_{\epsilon}^{-\frac12})^{(j)}|\lesssim
    h_{\epsilon}^{-\frac12}
    \lesssim
    r^{-\frac n2}
  \end{equation*}
  since the $\sigma_{j}$ are bounded, and this implies (5).

  Note that condition (i) can be relaxed to
  $h \gtrsim r+r^{n-1}$ when $n$ is even, and that
  in condition (iii) we could allow some singularity at
  0, for instance it is sufficient to assume that
  $|p_{3}|\lesssim \epsilon r^{-3}$.
\end{proof}

When the metric $h(r)$ has polynomial growth, it is more
convenient to use the following set of conditions to check
for admissibility:

\begin{proposition}[]\label{pro:crit3}
  Let $h,h_{\epsilon}:\mathbb{R}^{+}\to \mathbb{R}^{+}$
  be restrictions of smooth odd functions with
  $h'(0)=h_{\epsilon}'(0)=1$.
  Assume that
  \begin{enumerate}
  [noitemsep,topsep=0pt,parsep=0pt,partopsep=0pt,
  label=\textit{(\roman*)}]
    \item $h_{\epsilon}(r)\gtrsim cr$ for $r>0$ and some $c>0$
    \item $|h_{\epsilon}^{(j)}|\lesssim h_{\epsilon}r^{-j}$
      for $j\le[\frac{n-1}{2}]+2$ and $r\gg1$
    \item $|h^{(j)}-h_{\epsilon}^{(j)}|/h\le
      \epsilon r^{-j}$ for $j\le3$ and $r>0$
    \item $|h^{(j)}-h_{\epsilon}^{(j)}|/h\lesssim r^{-1}$
      for $1\le j\le[\frac{n-1}{2}]+2$ and $r\gg1$.
  \end{enumerate}
  If $h$ is admissible, then
  $h_{\epsilon}$ is also admissible, provded $\epsilon$
  is small enough.
\end{proposition}

\begin{proof}%[of ...]
  The proof is very similar to the previous one and again
  based on Proposition \ref{pro:crit1}.
  Note in particular that we have now
  \begin{equation*}
    |(h_{\epsilon}^{-\frac12})^{(j)}|
    \lesssim
    \sum
    h_{\epsilon}^{-\frac12}
    r^{-j_{1}-\cdots-j_{\nu}}
    \lesssim
    r^{-\frac12-j}
  \end{equation*}
  by assumptions (i), (ii).
\end{proof}

\subsection{Asymptotically flat manifolds}\label{sub:almost_flat}

For the flat metric $h(r)=r$ we have
$H(r)=\frac{(n-1)(n-3)}{4r^{2}}$ and it is elementary to check
that all conditions of Definition \ref{def:admman} are
satisfied, so that flat $\mathbb{R}^{n}$ is an admissible
manifold. Consider now a rotationally symmetric manifold
$M^{n}$ whose metric is a perturbation of the flat space,
of the form
\begin{equation}\label{eq:flat}
  h_{\epsilon}(r)=r+ \mu(r)
\end{equation}
with $\mu(r)$ odd, smooth, with $\mu'(0)=0$,
satisfying the following assumptions:
\begin{equation}\label{eq:assmu1}
  |\mu(r)|+r|\mu'(r)|+r^{2}|\mu''(r)|+r^{3}|\mu'''(r)|\le \epsilon r
  \quad\text{for all}\quad
  r>0
\end{equation}
and
\begin{equation}\label{eq:assmu2}
  \textstyle
  |\mu^{(j)}(r)|\lesssim r^{1-j}
  \quad\text{for}\quad
  r\gg1,
  \quad
  j\le[\frac{n-1}{2}]+2.
\end{equation}
Then Proposition \ref{pro:crit3} implies immediately
that the metric $h_{\epsilon}$ is admissible if
$\epsilon$ is sufficiently small.

Note that in dimension $n=4$ this result is essentially a
corollary of Theorem 1.1 in \cite{Lawrie12-a}.
In that paper the global existence of small wave
maps is proved on four dimensional,
asymptotically flat manifolds
without symmetry assumptions.

\subsection{Perturbations of hyperbolic spaces}\label{sub:expgr}

For real hyperbolic spaces $M=\mathbb{H}^{n}$ we have
\begin{equation*}
  h(r)=\sinh(r),
  \qquad
  h_{\infty}=\frac{(n-1)^{2}}{4},
  \qquad
  H(r)=h_{\infty}+\frac{(n-1)(n-3)}{4\sinh^{2}r}
\end{equation*}
so that (i), (ii)
of Definition \ref{def:admman} are satisfied. Moreover
\begin{equation*}
  P(r)=\frac{(n-1)(n-3)}{4} \frac{r}{\sinh^{2} r}+
       \frac{1-\delta_{0}}{4r}
\end{equation*}
and it is easy to check that $P\ge0\ge P'$ if
$\delta_{0}<1$. Thus Corollary \ref{cor:hrfirst} implies
global existence of equivariant wave maps
$\phi:\mathbb{R}\times \mathbb{H}^{n}\to N^{\ell}$
for small data in the critical space $H^{\frac n2}$,
for $n\ge3$. Note that this result
could be also obtained by applying the sharp Strichartz
estimates in
\cite{AnkerPierfelice09-a},
\cite{AnkerPierfeliceVallarino12-a}
for linear flows on
real hyperbolic spaces.

However, we are able to treat the more general case of
a perturbation of the hyperbolic metric
\begin{equation}\label{eq:hyp}
  h_{\mu}(r)=\sinh r+\mu(r)
\end{equation}
under rather bland conditions on the perturbation $\mu(r)$.
For instance, we may assume that, for all $r>0$,
\begin{equation*}
  |\mu(r)|+|\mu'(r)|+|\mu''(r)|+|\mu'''(r)|\le
  \epsilon \bra{r}^{-3}\sinh r
\end{equation*}
and that, for sufficiently large $r\gg1$,
\begin{equation*}
  |\mu^{(j)}(r)|\le C r^{-1}e^{r}.
\end{equation*}
Note that the perturbation $\mu(r)$ can be unbounded
as $r\to+\infty$.
Then, by Proposition \ref{pro:crit2},
the function $h_{\mu}$ is also admissible
provided $\epsilon$ is small enough, and we obtain the
global existence of small equivariant wave maps of critical
regularity from a base manifold with
metric $dr^{2}+h_{\mu}(r)^{2}d \omega^{2}_{\mathbb{S}^{n-1}}$.

\subsection{Manifolds with prescribed growth}\label{sub:polynmanif}

It is not difficult to construct examples of
admissible manifolds if one allows a singularity at the south
pole $r=0$. The singularity can be smoothed out by
flattening the manifold near zero with a suitable cutoff.
We illustrate the procedure with two examples, one metric with
polynomial growth and one with exponential growth.

The choice
\begin{equation*}
  h(r)=r(1+\sqrt{r})^{M},
  \qquad
  M>0
\end{equation*}
gives rise to
\begin{equation*}
  H(r)=\frac{n-1}{16(1+\sqrt{r})^{2}r^{2}}H_{0}(r)
\end{equation*}
with
\begin{equation*}
  H_{0}(r)=
  4 (n-3) + (4n ( M+2)-6 ( M+4) ) \sqrt{r}
    + (M+2) (M (n-1)+2 (n-3)) r
\end{equation*}
and
\begin{equation*}
  P(r)=
  \frac{Q_{0}+Q_{1}\sqrt{r}+Q_{2}r}{16(1+\sqrt{r})^{2}r},
\end{equation*}
where
\begin{equation*}
  Q_{0}=4(n-2)^{2}-4\delta_{0},
  \qquad
  Q_{1}=2M(n-1)(2n-3)+8(n-2)^{2}-8\delta_{0},
\end{equation*}
\begin{equation*}
  Q_{2}=( M n+2n-M-4)^2-4\delta_{0}.
\end{equation*}
Since $0<\delta_{0}<1$,
we see that $P(r)$ is a sum of positive, decreasing
functions, and all properties of
Definition \ref{def:admman} are satisfied, with the
exception of $h(r)$ being the restriction of a smooth
odd function, due to the singularity at $r=0$.
Now we modify the definition of $h(r)$
as follows:
\begin{equation*}
  h_{\epsilon}(r)=r(1+\sqrt{r}\cdot e^{-\frac \epsilon r})^{M}
\end{equation*}
where $\epsilon>0$ is a small parameter.
Then $h_{\epsilon}$ is still admissible provided
$\epsilon$ is small enough, as it follows from
Proposition \ref{pro:crit3},
and is of course smooth at $r=0$
and can be extended to an odd function on $\mathbb{R}$.

In a similar way, the metric with exponential growth
\begin{equation*}
  h(r)=e^{r}-1
\end{equation*}
gives
\begin{equation*}
  H(r)=h_{\infty}+\frac{(n-1) (2 (n-2) e^{r}-n+1)}{4 (e^{r}-1)^2},
  \qquad
  h_{\infty}=\frac{(n-1)^{2}}{4}
\end{equation*}
so that
\begin{equation*}
  P(r)=
  % \frac{(2 e^{r} (n-2) - n+1) (n-1) r}{4 (e^{r}-1)^2}
  % +\frac{1-\delta_{0}}{4r}=
  \frac{(n-1)(n-2)}{2}\frac{r}{e^{r}-1}
  +\frac{(n-1)(n-3)}{4}\frac{r}{(e^{r}-1)^{2}}
  +\frac{1-\delta_{0}}{4r}
\end{equation*}
is a sum of positive, decreasing terms. Thus all conditions
in Definition \ref{def:admman} are trivially satisfied,
with the exception of $h(r)$ being the restriction of an
odd smooth function, but it is sufficient to modify it
near zero as above to obtain an admissible metric.

% b_postamble
% \printbibliography
% \bibliographystyle{plain}
% \bibliography{/Users/piero/Documents/Biblioteca/-bib/bibliodatabase.bib}

\end{document}